\documentclass{amsart}
\usepackage{amsmath,amssymb,mathrsfs}
\usepackage{amsthm,hyperref}

\usepackage{graphicx}

\date{}
\allowdisplaybreaks

\newtheorem{theo}{Theorem}[section]
\newtheorem{lemm}{Lemma}[section]
\newtheorem{coro}{Corollary}[section]
\newtheorem{prop}{Proposition}[section]

\theoremstyle{remark}
\newtheorem{rema}{Remark}[section]

\theoremstyle{definition}
\newtheorem{defi}{Definition}[section]

\graphicspath{{./pictures/}}	

\author {Ivan Dynnikov}
\address{\noindent V.A.Steklov Mathematical Institute of Russian Academy of Science, 8 Gubkina Str., Moscow 119991, Russia}
\address{\noindent St.\ Petersburg State University, Line 14th (Vasilyevsky Island), 29, Saint Petersburg, 199178, Russia}
\thanks{This work is supported by the Russian Science Foundation under grant~19-11-00151.}

\email{dynnikov@mech.math.msu.su}

\title{Transverse-Legendrian links}

\begin{document}

\begin{abstract}
In recent joint works of the present author with M.\,Prasolov and V.\,Shastin a new
technique for distinguishing Legendrian knots has been developed. In this paper
the technique is extended further to provide a tool for distinguishing transverse knots.
It is shown that the equivalence problem for transverse knots with
trivial orientation-preserving symmetry group
is algorithmically solvable. In a future paper
the triviality condition for the orientation-preserving symmetry group will be dropped.
\end{abstract}

\maketitle

\section{Introduction}
Rectangular (or grid) diagrams of links provide a convenient combinatorial framework for
studying Legendrian and transverse links. Namely, there are the following naturally defined bijections,
each respecting the topological type of the link:
$$\begin{aligned}
\mathscr R/\bigl\langle\overrightarrow{\mathrm I},\overleftarrow{\mathrm I}\bigr\rangle&\cong
\bigl\{\xi_+\text{-Legendrian link types}\big\},\\
\mathscr R/\bigl\langle\overrightarrow{\mathrm{II}},\overleftarrow{\mathrm{II}}\bigr\rangle&\cong
\bigl\{\xi_-\text{-Legendrian link types}\big\},\\
\mathscr R/\bigl\langle\overrightarrow{\mathrm I},\overleftarrow{\mathrm I},\overleftarrow{\mathrm{II}}\bigr\rangle\cong
\mathscr R/\bigl\langle\overrightarrow{\mathrm I},\overleftarrow{\mathrm I},\overrightarrow{\mathrm{II}}\bigr\rangle&\cong
\bigl\{\xi_+\text{-transverse link types}\big\},\\
\mathscr R/\bigl\langle\overrightarrow{\mathrm I},\overleftarrow{\mathrm{II}},\overrightarrow{\mathrm{II}}\bigr\rangle\cong
\mathscr R/\bigl\langle\overleftarrow{\mathrm I},\overleftarrow{\mathrm{II}},\overrightarrow{\mathrm{II}}\bigr\rangle&\cong
\bigl\{\xi_-\text{-transverse link types}\big\},\\
\end{aligned}$$
where~$\mathscr R/\langle T_1,\ldots,T_k\rangle$ means `oriented rectangular diagrams viewed up
to exchange moves and (de)sta\-bi\-li\-za\-tions of oriented types~$T_1,\ldots,T_k$' (we use
the notation of~\cite{bypasses} for the oriented types of stabilizations and destabilizations; see also Definition~\ref{moves-def} and Figure~\ref{stab-fig} below),
$\xi_+$ is the standard contact structure of~$\mathbb S^3$, and~$\xi_-$ is the mirror image of~$\xi_+$. A proof of these facts can be found in~\cite{OST}.

With the notation above at hand, the elements of the sets
$$\mathscr R/\bigl\langle\overrightarrow{\mathrm I},\overrightarrow{\mathrm{II}}\bigr\rangle\cong
\mathscr R/\bigl\langle\overleftarrow{\mathrm I},\overleftarrow{\mathrm{II}}\bigr\rangle\cong
\mathscr R/\bigl\langle\overrightarrow{\mathrm I},\overleftarrow{\mathrm{II}}\bigr\rangle\cong
\mathscr R/\bigl\langle\overleftarrow{\mathrm I},\overrightarrow{\mathrm{II}}\bigr\rangle$$
are naturally interpreted
as braids viewed up to conjugacy and Birman--Menasco exchange moves defined in~\cite{bm4}
(these entities are called Birman--Menasco classes in~\cite{bypasses}),
and elements of
$$\mathscr R/\bigl\langle\overrightarrow{\mathrm I},\overleftarrow{\mathrm I},\overrightarrow{\mathrm{II}},
\overleftarrow{\mathrm{II}}\bigr\rangle$$
as topological types of oriented links in~$\mathbb S^3$, see~\cite{cromwell,simplification}.

The elements of~$\mathscr R/\langle\varnothing\rangle$ are so called exchange classes, which mean rectangular diagrams viewed up
to exchange moves. The number of possible combinatorial types of diagrams in each exchange class
is finite, so the equivalence problem for exchange classes is trivially decidable. This fact and the results of~\cite{distinguishing}
are used in~\cite{dyn-shast}
to solve the equivalence problem for Legendrian knots of topological types having trivial orientation-preserving symmetry group.
It is noted in~\cite{dyn-shast} that the equivalence problem for transverse knots of the same topological types
can be solved in a similar manner, once we are able to solve the equivalence problem
for the elements of~$\mathscr R/\bigl\langle\overrightarrow{\mathrm I}\bigr\rangle$,
$\mathscr R/\bigl\langle\overleftarrow{\mathrm I}\bigr\rangle$, $\mathscr R/\bigl\langle\overrightarrow{\mathrm{II}}\bigr\rangle$,
and $\mathscr R/\bigl\langle\overleftarrow{\mathrm{II}}\bigr\rangle$
(see~\cite[Remark~7.1]{dyn-shast}).

In this note we give a topological interpretation to the elements of these sets
and solve the equivalence problem for them, thus extending the method of~\cite{distinguishing,dyn-shast} to transverse knots.
\section{Notation}

We denote by~$\mathbb T^2$ the two-dimensional torus~$\mathbb S^1\times\mathbb S^1$, and by~$\theta,\varphi$
the angular coordinates on~$\mathbb T^2$, which run through~$\mathbb R/(2\pi\mathbb Z)$. Denote
by~$p_\theta$ and~$p_\varphi$ the projection maps from~$\mathbb T^2$ to the first and the second~$\mathbb S^1$-factors, respectively.

We regard the three-sphere~$\mathbb S^3$ as the join~$\mathbb S^1*\mathbb S^1$ of two circles,
and use the associated coordinate system $\theta,\varphi,\tau$:
$$\mathbb S^3=\mathbb S^1\times\mathbb S^1\times[0;1]/\bigl((\theta',\varphi,0)\sim(\theta'',\varphi,0),
(\theta,\varphi',1)\sim(\theta,\varphi'',1)\ \forall\theta,\theta',\theta'',\varphi,\varphi',\varphi''\in\mathbb S^1\bigr).$$
(Observe that~$\tau$ is set to~$1$ on the first copy of~$\mathbb S^1$, on which the angular coordinate is~$\theta$,
and to~$0$ on the second one, where the angular coordinate is~$\varphi$.)

The map~$p_{\theta,\varphi}:\mathbb S^3\setminus\bigl(\mathbb S^1_{\tau=1}\cup\mathbb S^1_{\tau=0}\bigr)\rightarrow\mathbb T^2$
defined by~$p_{\theta,\varphi}(\theta,\varphi,\tau)=(\theta,\varphi)$ is referred to as the \emph{torus projection}.

For two distinct points~$x_1,x_2\in\mathbb S^1$ we denote by~$[x_1;x_2]$ (respectively, $(x_1;x_2)$) the closed (respectively, open) interval in
$\mathbb S^1$ starting at~$x_1$ and ending at~$x_2$.

\section{Rectangular diagrams of links}

\begin{defi}
By \emph{an oriented rectangular diagram of a link} we mean a non-empty finite subset~$R\subset\mathbb T^2$
with a decomposition~$R=R^+\sqcup R^-$ into disjoint union of two subsets~$R^+$ and~$R^-$ such
that we have~$p_\theta(R^+)=p_\theta(R^-)$, $p_\varphi(R^+)=p_\varphi(R^-)$, and each of~$p_\theta$, $p_\varphi$
restricted to each of~$R^+$, $R^-$ is injective.

The elements of~$R$ (respectively, of~$R^+$ or $R^-$)
are called \emph{vertices} (respectively, \emph{positive vertices} or \emph{negative vertices}) of~$R$.

Pairs~$(u,v)$ of vertices of~$R$ such that~$p_\theta(u)=p_\theta(v)$ (respectively, $p_\varphi(u)=p_\varphi(v)$)
are called \emph{vertical} (respectively, \emph{horizontal}) \emph{edges} of~$R$.
\end{defi}

With every oriented rectangular diagram~$R$ of a link we define \emph{the associated oriented link}~$\widehat R\subset\mathbb S^3$
as the closure of the preimage~$p_{\theta,\varphi}^{-1}(R)$ oriented so that~$\tau$ increases on
the oriented arcs constituting~$p_{\theta,\varphi}^{-1}(R^+)$ and decreases on
the oriented arcs constituting~$p_{\theta,\varphi}^{-1}(R^-)$.

A planar diagram of a link topologically equivalent to~$\widehat R$ can be obtained as follows. Cut the torus~$\mathbb T^2$
along a longitude and a meridian not passing through a vertex of~$R$ to obtain a square. Connect
the vertices in every edge by a vertical or horizontal straight line segment and make all verticals overpasses
at all crossings. Orient the obtained diagram so that each vertical edge is directed
from a positive vertex to a negative one, and each horizontal edge from a negative to a positive one.
For an example see Figure~\ref{rd-fig}, where positive vertices are black and negative ones are white.
\begin{figure}[ht]
\begin{tabular}{ccc}
\includegraphics[width=150pt]{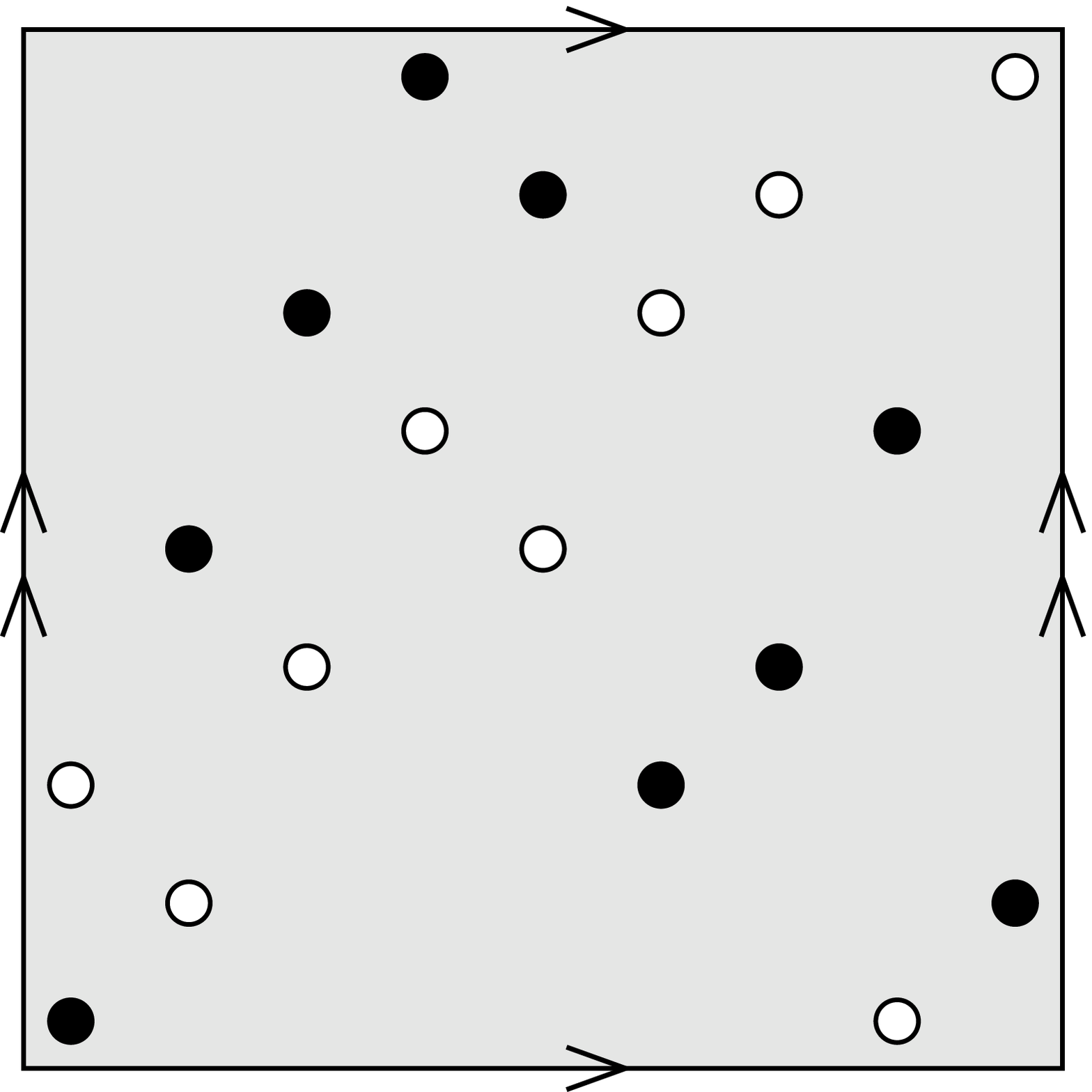}\put(-80,20){$\mathbb T^2$}&\hbox to 2cm{\hss}&\includegraphics[width=150pt]{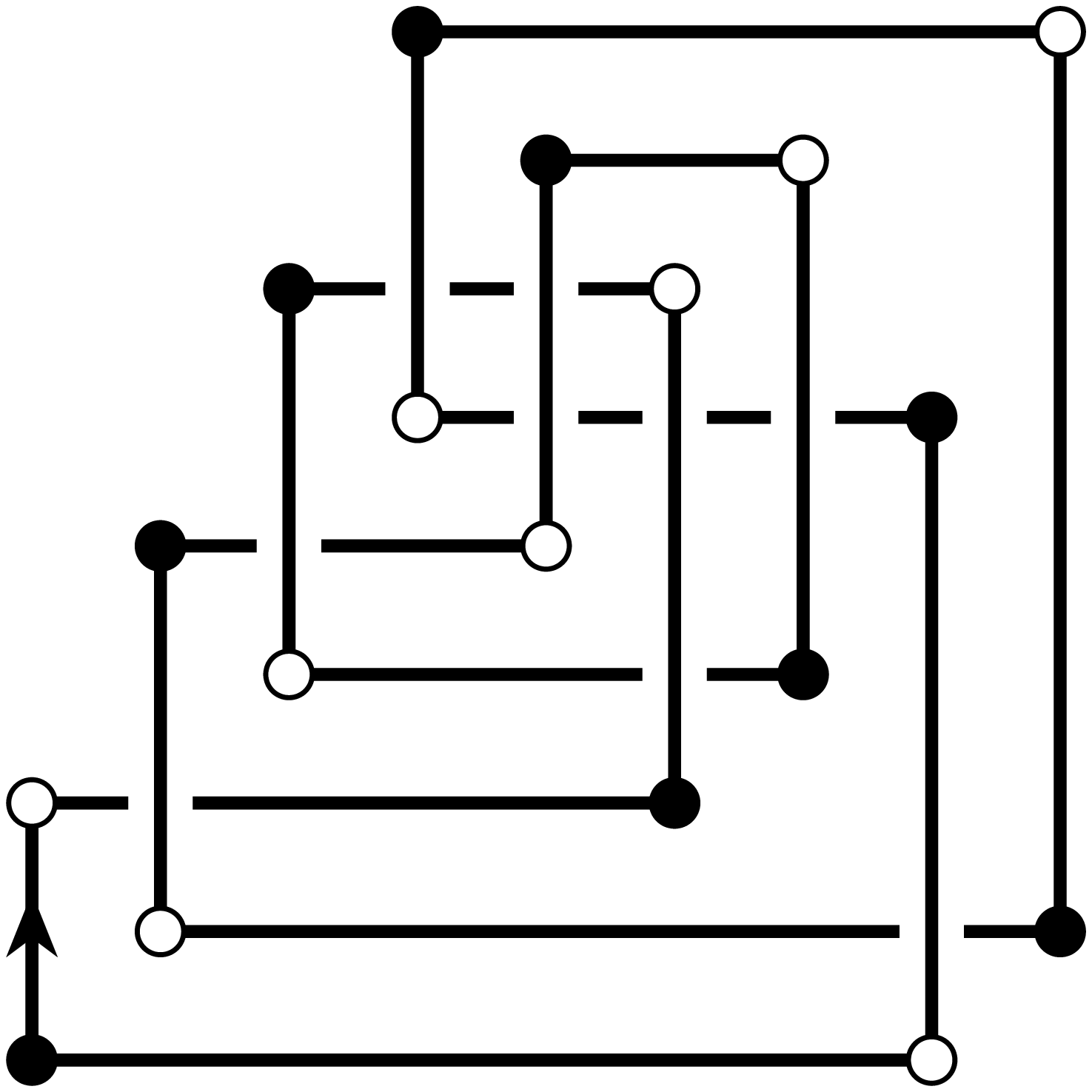}\\
$R$&&a knot equivalent to~$\widehat R$
\end{tabular}
\caption{A rectangular diagram of a link and a planar diagram of the corresponding link}\label{rd-fig}
\end{figure}

In this paper, all links and their diagrams are assumed to be oriented, so we omit `oriented' in the sequel.

\begin{defi}\label{moves-def}
Let~$R_1$ and~$R_2$ be rectangular diagrams of a link such that,
for some~$\theta_1,\theta_2,\varphi_1,\varphi_2\in\mathbb S^1$, the following holds:
\begin{enumerate}
\item
$\theta_1\ne\theta_2$, $\varphi_1\ne\varphi_2$;
\item
the symmetric difference~$R_1\triangle R_2$ is~$\{\theta_1,\theta_2\}\times\{\varphi_1,\varphi_2\}$;
\item
the intersection of the rectangle~$[\theta_1;\theta_2]\times[\varphi_1;\varphi_2]$
with~$R_1\cup R_2$ coincides with~$R_1\triangle R_2$;
\item
one, two, or three consecutive corners of the rectangle~$[\theta_1;\theta_2]\times[\varphi_1;\varphi_2]$
belong to~$R_1$, and the other(s) to~$R_2$;
\item
the orientations of~$R_1$ and~$R_2$ agree on~$R_1\cap R_2$, which
means~$R_1^+\cap R_2=R_1\cap R_2^+$ (equivalently, $R_1^-\cap R_2=R_1\cap R_2^-$).
\end{enumerate}
Then we say that the passage~$R_1\mapsto R_2$ is \emph{an elementary move}.

An elementary move~$R_1\mapsto R_2$ is called:
\begin{itemize}
\item
\emph{an exchange move} if~$|R_1|=|R_2|$,
\item
\emph{a stabilization move} if~$|R_2|=|R_1|+2$, and
\item
\emph{a destabilization move} if~$|R_2|=|R_1|-2$,
\end{itemize}
where~$|R|$ denotes the number of vertices of~$R$.
\end{defi}

We distinguish two \emph{types} and four \emph{oriented types} of stabilizations and destabilizations as follows.

\begin{defi}
Let~$R_1\mapsto R_2$ be a stabilization, and let~$\theta_1,\theta_2,\varphi_1,\varphi_2$ be as in Definition~\ref{moves-def}.
Denote by~$v$ an element of~$R_1\cap([\theta_1;\theta_2]\times[\varphi_1;\varphi_2])$, which is unique.
We say that the stabilization~$R_1\mapsto R_2$ and the destabilization~$R_2\mapsto R_1$
are of \emph{type~\rm I} (respectively, of \emph{type~\rm II}) if
$v\in\{(\theta_1,\varphi_1),(\theta_2,\varphi_2)\}$
(respectively, $v\in\{(\theta_1,\varphi_2),(\theta_2,\varphi_1)\}$).

Let~$\varphi_0\in\{\varphi_1,\varphi_2\}$ be such that~$\{\theta_1,\theta_2\}\times\{\varphi_0\}\subset R_2$.
The stabilization~$R_1\mapsto R_2$ and the destabilization~$R_2\mapsto R_1$
are of \emph{oriented type~$\overrightarrow{\mathrm I}$}
(respectively, of \emph{oriented type~$\overrightarrow{\mathrm{II}}$}) if they are of type~I (respectively, of type~II),
and~$(\theta_2,\varphi_0)$ is a positive vertex of~$R_2$.
The stabilization~$R_1\mapsto R_2$ and the destabilization~$R_2\mapsto R_1$
are of \emph{oriented type~$\overleftarrow{\mathrm I}$}
(respectively, of \emph{oriented type~$\overleftarrow{\mathrm{II}}$}) if they are of type~I (respectively, of type~II)
and~$(\theta_2,\varphi_0)$ is a negative vertex of~$R_2$.
\end{defi}

Elementary moves are illustrated in Figures~\ref{exch-fig} and~\ref{stab-fig}, where the shaded rectangle is supposed
to contain no vertices of the diagrams except the indicated ones.
\begin{figure}[ht]
\begin{tabular}{ccccccc}
\includegraphics[height=60pt]{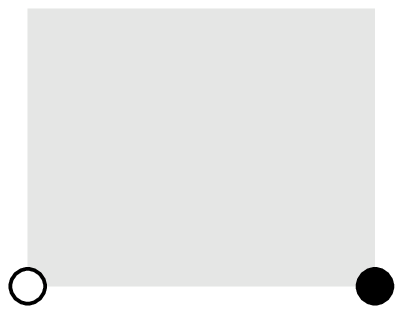}&\raisebox{27pt}{$\leftrightarrow$}&
\includegraphics[height=60pt]{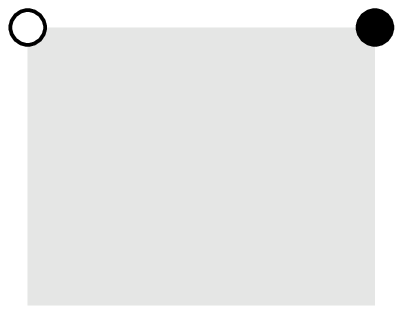}&\hbox to 40 pt{\hss}&
\includegraphics[height=60pt]{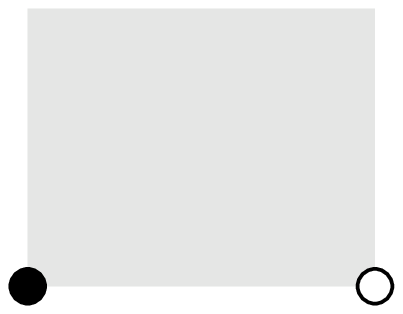}&\raisebox{27pt}{$\leftrightarrow$}&
\includegraphics[height=60pt]{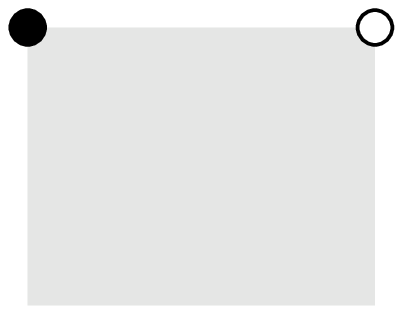}\\[20pt]
\includegraphics[height=60pt]{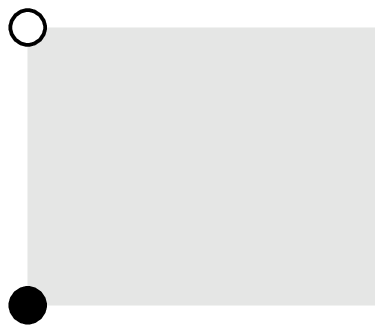}&\raisebox{27pt}{$\leftrightarrow$}&
\includegraphics[height=60pt]{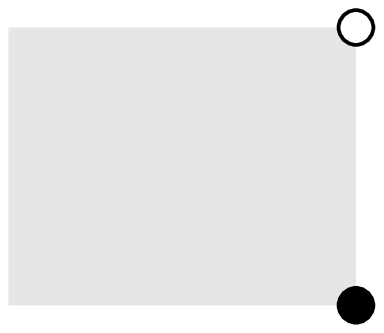}&\hbox to 40 pt{\hss}&
\includegraphics[height=60pt]{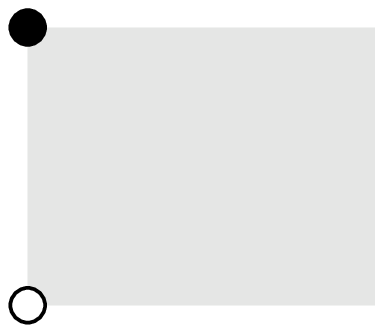}&\raisebox{27pt}{$\leftrightarrow$}&
\includegraphics[height=60pt]{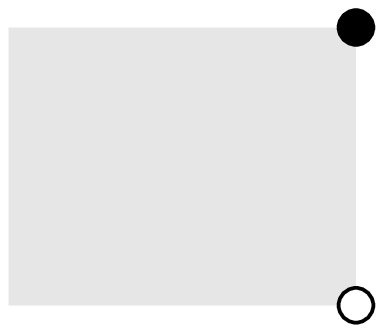}\\
\end{tabular}
\caption{Exchange moves}\label{exch-fig}
\end{figure}
\begin{figure}[ht]
\begin{tabular}{ccccccc}
\includegraphics[height=60pt]{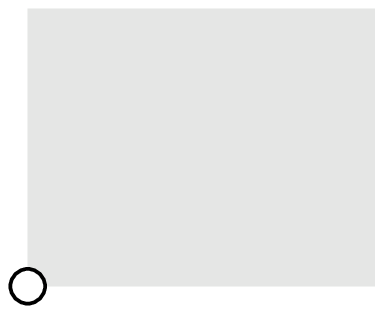}&\raisebox{27pt}{$\leftrightarrow$}&
\includegraphics[height=60pt]{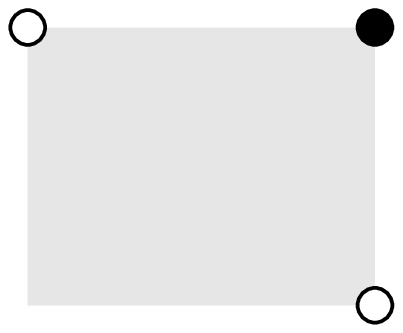}&\hbox to 40 pt{\hss}&
\includegraphics[height=60pt]{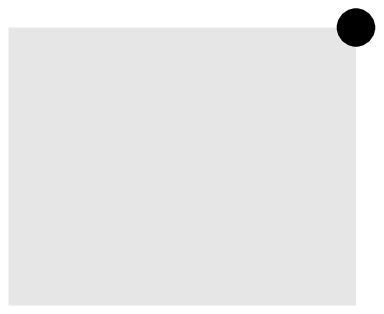}&\raisebox{27pt}{$\leftrightarrow$}&
\includegraphics[height=60pt]{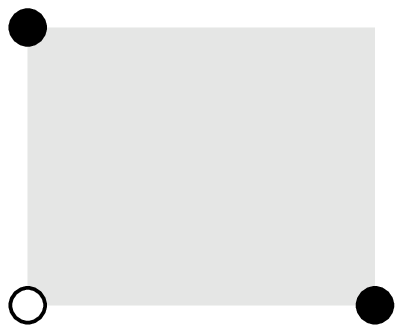}\\
&\hbox to 0 pt{\hss type $\overrightarrow{\mathrm I}$\hss}&
&&&\hbox to 0 pt{\hss type $\overrightarrow{\mathrm I}$\hss}\\[20pt]
\includegraphics[height=60pt]{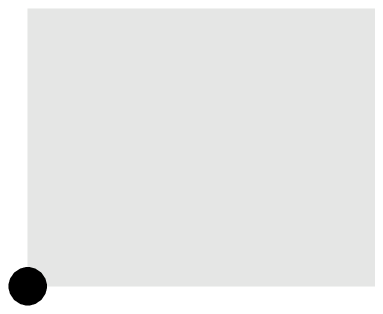}&\raisebox{27pt}{$\leftrightarrow$}&
\includegraphics[height=60pt]{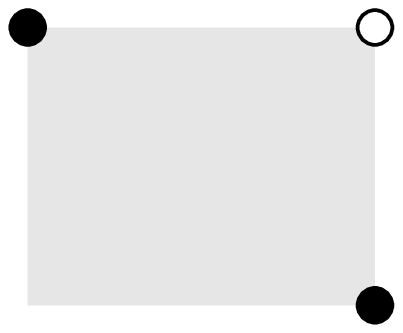}&\hbox to 40 pt{\hss}&
\includegraphics[height=60pt]{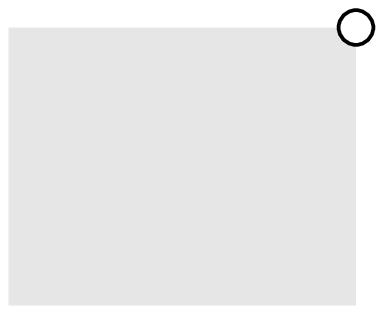}&\raisebox{27pt}{$\leftrightarrow$}&
\includegraphics[height=60pt]{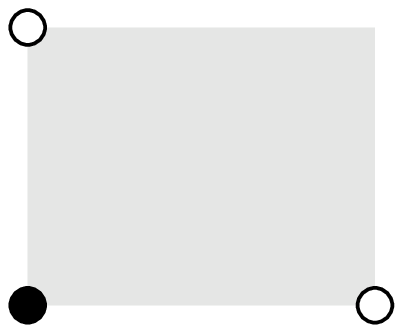}\\
&\hbox to 0 pt{\hss type $\overleftarrow{\mathrm I}$\hss}&
&&&\hbox to 0 pt{\hss type $\overleftarrow{\mathrm I}$\hss}\\[20pt]
\includegraphics[height=60pt]{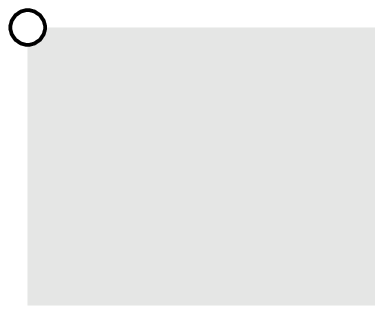}&\raisebox{27pt}{$\leftrightarrow$}&
\includegraphics[height=60pt]{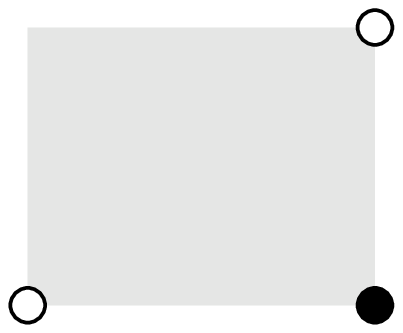}&\hbox to 40 pt{\hss}&
\includegraphics[height=60pt]{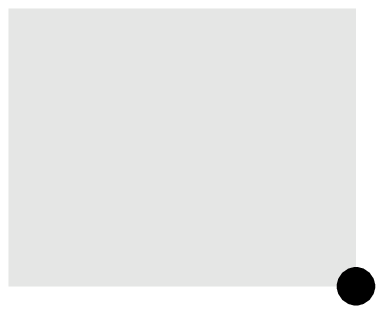}&\raisebox{27pt}{$\leftrightarrow$}&
\includegraphics[height=60pt]{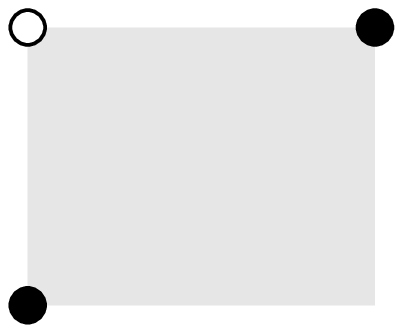}\\
&\hbox to 0 pt{\hss type $\overrightarrow{\mathrm{II}}$\hss}&
&&&\hbox to 0 pt{\hss type $\overrightarrow{\mathrm{II}}$\hss}\\[20pt]
\includegraphics[height=60pt]{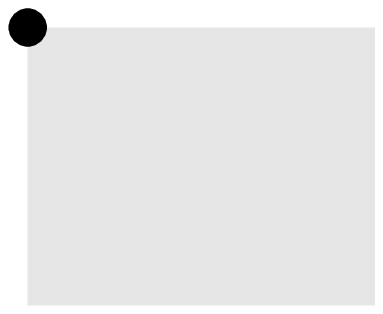}&\raisebox{27pt}{$\leftrightarrow$}&
\includegraphics[height=60pt]{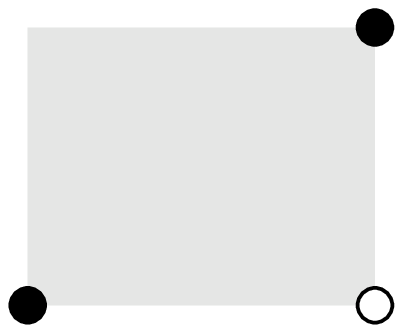}&\hbox to 40 pt{\hss}&
\includegraphics[height=60pt]{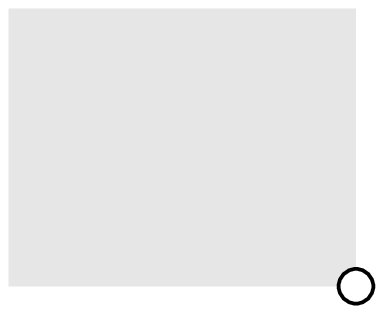}&\raisebox{27pt}{$\leftrightarrow$}&
\includegraphics[height=60pt]{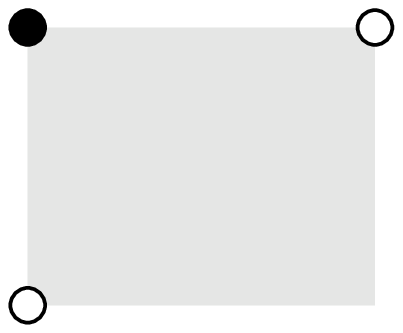}\\
&\hbox to 0 pt{\hss type $\overleftarrow{\mathrm{II}}$\hss}&
&&&\hbox to 0 pt{\hss type $\overleftarrow{\mathrm{II}}$\hss}
\end{tabular}
\caption{Stabilization and destabilization moves}\label{stab-fig}
\end{figure}

The set of all rectangular diagrams of links will be denoted by~$\mathscr R$. For any subset~$\{T_1,\ldots,T_k\}$
of~$\{\overrightarrow{\mathrm I},\overleftarrow{\mathrm I},\overrightarrow{\mathrm{II}},\overleftarrow{\mathrm{II}}\}$,
we denote by~$\langle T_1,\ldots,T_k\rangle$ the equivalence relation on~$\mathscr R$ generated
by all stabilizations and destabilizations of oriented types~$T_1,\ldots,T_k$ and exchange moves.
For a rectangular diagram of a link~$R\in\mathscr R$, we denote by~$[R]_{T_1,\ldots,T_k}$
the equivalence class of~$R$ in~$\mathscr R/\langle T_1,\ldots,T_k\rangle$.

The following statement is nearly a reformulation of~\cite[Proposition on page~42 + Theorem on page~45]{cromwell}
and~\cite[Proposition~4]{simplification} (the three versions use slightly different
settings and sets of moves, but their equivalence is easily seen).

\begin{theo}
The map
$$R\mapsto\text{the topological type of }\widehat R$$
establishes a one-to-one correspondence between $\mathscr R/\langle
\overrightarrow{\mathrm I},\overleftarrow{\mathrm I},\overrightarrow{\mathrm{II}},\overleftarrow{\mathrm{II}}\rangle$
and the set of all link types.
\end{theo}

\section{Decidability for the equivalence of transverse knots}

Here is the main technical result of the present paper:

\begin{theo}\label{main-tech-theo}
For any~$T\in\{\overrightarrow{\mathrm I},\overleftarrow{\mathrm I},\overrightarrow{\mathrm{II}},\overleftarrow{\mathrm{II}}\}$
there is an algorithm for deciding, given two rectangular diagrams of a link~$R_1$, $R_2$, whether or not~$[R_1]_T=[R_2]_T$.
\end{theo}

To prove Theorem~\ref{main-tech-theo} we need some preparations.

For a rectangular diagram of a link~$R$, denote by~$\Gamma_{\overrightarrow{\mathrm{II}}}(R)$
the following union of closed immersed staircase-like curves in~$\mathbb T^2$:
$$\Gamma_{\overrightarrow{\mathrm{II}}}(R)=\left(\bigcup_{(\theta_0,\varphi_1)\in R^+,\
(\theta_0,\varphi_2)\in R^-}\{\theta_0\}\times[\varphi_1;\varphi_2]\right)\cup
\left(\bigcup_{(\theta_1,\varphi_0)\in R^-,\
(\theta_2,\varphi_0)\in R^+}[\theta_1;\theta_2]\times\{\varphi_0\}\right)$$
oriented by demanding that~$\theta+\varphi$ locally increase on every straight line segment
in this union. These straight line segments will be referred to as \emph{the edges} of~$\Gamma_{\overrightarrow{\mathrm{II}}}(R)$.
Thus, the pair of endpoints of an edge of~$\Gamma_{\overrightarrow{\mathrm{II}}}(R)$ is an edge of~$R$, and vice versa.

An example is shown in Figure~\ref{gamma-fig}.
\begin{figure}[ht]
\begin{tabular}{ccc}
\includegraphics[width=150pt]{rd0.eps}\put(-80,20){$\mathbb T^2$}&\hbox to 2cm{\hss}&\includegraphics[width=150pt]{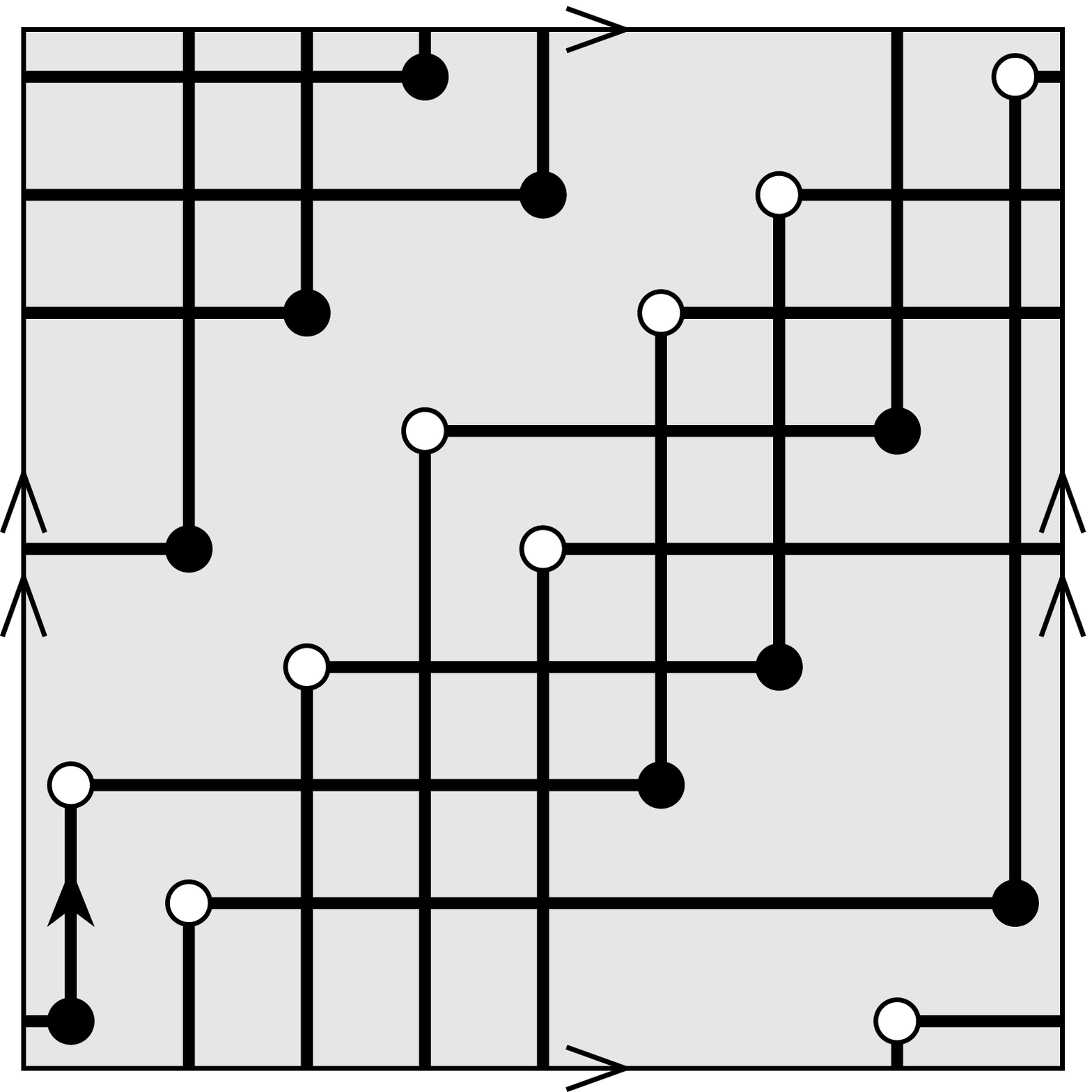}\\
$R$&&$\Gamma_{\overrightarrow{\mathrm{II}}}(R)$
\end{tabular}
\caption{A rectangular diagram of a link~$R$ and the curve~$\Gamma_{\protect\overrightarrow{\mathrm{II}}}(R)$}\label{gamma-fig}
\end{figure}
The union of curves~$\Gamma_{\overrightarrow{\mathrm{II}}}(R)$ can also be described as the torus projection
of the link~$\widehat R_\varepsilon$ obtained from~$\widehat R$ by replacing each arc in the domain~$\tau\in[0;\varepsilon]$
(respectively, $\tau\in[1-\varepsilon;1]$)
by an arc on which the coordinates~$\varphi$ and~$\tau$ are constant, and~$\theta$ is increasing
(respectively, $\theta$ and~$\tau$ are constant, and~$\varphi$ is increasing), where~$\varepsilon\in(0;1/2)$.

With every rectangular diagram of a link~$R$ we associate a triple of
numbers~$\omega_{\overrightarrow{\mathrm{II}}}(R)\in\mathbb N\times\mathbb N\times(\mathbb N\cup\{0\})$
as follows: $\omega_{\overrightarrow{\mathrm{II}}}(R)=(k,l,m)$, where~$m$ is the number of double points in~$\Gamma_{\overrightarrow{\mathrm{II}}}(R)$,
and~$(k,l)\in\mathbb Z^2=H_1(\mathbb T^2;\mathbb Z)$ is the homology class of~$\Gamma_{\overrightarrow{\mathrm{II}}}(R)$, that is,
$$k=\frac1{2\pi}\int\limits_{\Gamma_{\overrightarrow{\mathrm{II}}}(R)}d\theta,\qquad l=\frac1{2\pi}\int\limits_{\Gamma_{\overrightarrow{\mathrm{II}}}(R)}d\varphi.$$

\begin{lemm}\label{omega-inv-lem}
If~$[R]_{\overrightarrow{\mathrm{II}}}=[R']_{\overrightarrow{\mathrm{II}}}$, then~$\omega_{\overrightarrow{\mathrm{II}}}(R)=
\omega_{\overrightarrow{\mathrm{II}}}(R')$.
\end{lemm}

\begin{proof}
To simplify the notation we put~$\Gamma=\Gamma_{\overrightarrow{\mathrm{II}}}(R)$ and~$\Gamma'=\Gamma_{\overrightarrow{\mathrm{II}}}(R')$.
It suffices to consider the case when~$R\mapsto R'$ is an exchange move or a type~$\overrightarrow{\mathrm{II}}$ stabilization.
One can check that, for any of these moves,
the closure of the symmetric difference~$\Gamma\triangle\Gamma'$
is the boundary of a rectangle~$r$ (which is not necessarily the one mentioned in Definition~\ref{moves-def}),
with the bottom and right sides of~$r$ belonging to one of~$\Gamma$, $\Gamma'$, and the top and left sides
to the other. Moreover, if~$\Gamma$ and~$\Gamma'$ are viewed as $1$-chains, then~$\Gamma-\Gamma'=\partial r$
for some orientation of~$r$. The cases are sketched in Figure~\ref{allowed-fig}.
\begin{figure}[ht]
\begin{tabular}{ccccccc}
\includegraphics[height=50pt]{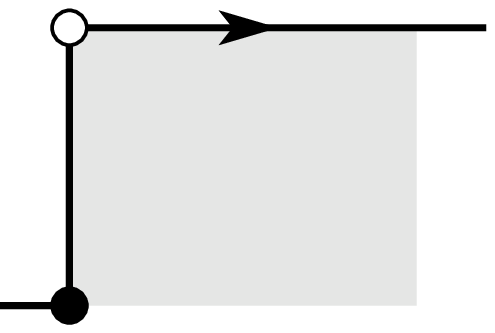}\put(-31,22){$r$}&\raisebox{23pt}{$\leftrightarrow$}&
\includegraphics[height=50pt]{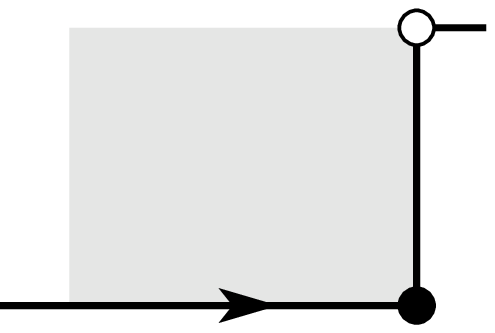}\put(-31,22){$r$}&\hbox to 40 pt{\hss}&
\includegraphics[height=50pt]{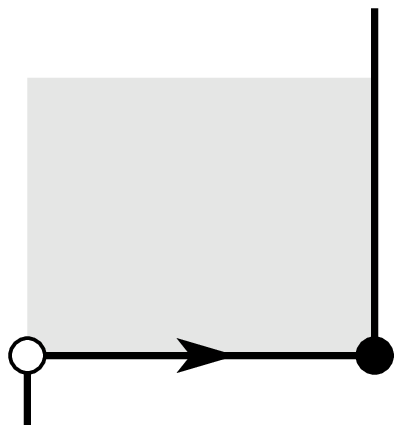}\put(-31,22){$r$}&\raisebox{23pt}{$\leftrightarrow$}&
\includegraphics[height=50pt]{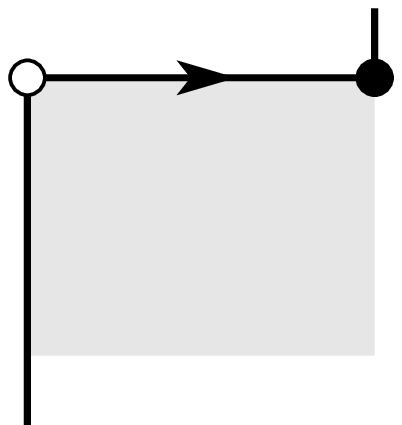}\put(-31,22){$r$}\\
\includegraphics[height=50pt]{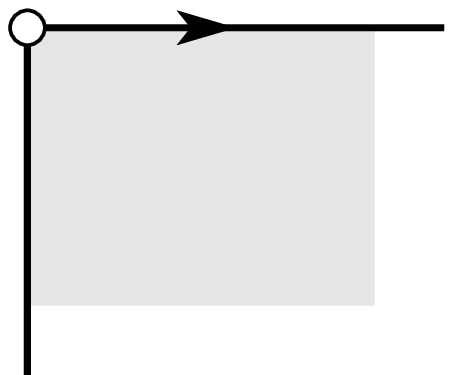}\put(-31,22){$r$}&\raisebox{23pt}{$\leftrightarrow$}&
\includegraphics[height=50pt]{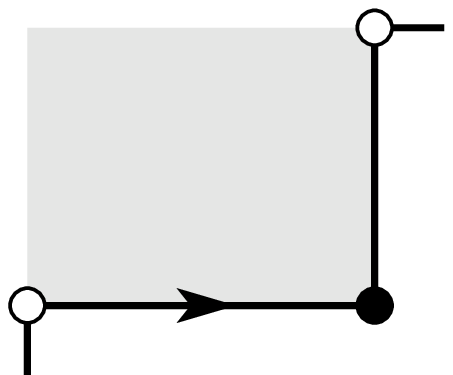}\put(-31,22){$r$}&\hbox to 40 pt{\hss}&
\includegraphics[height=50pt]{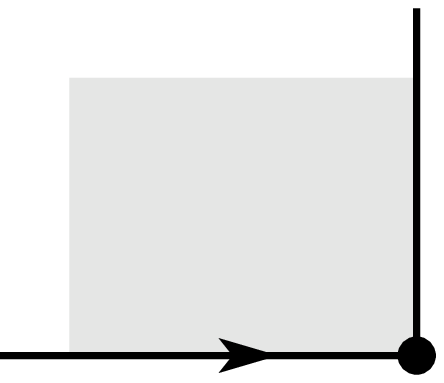}\put(-31,22){$r$}&\raisebox{23pt}{$\leftrightarrow$}&
\includegraphics[height=50pt]{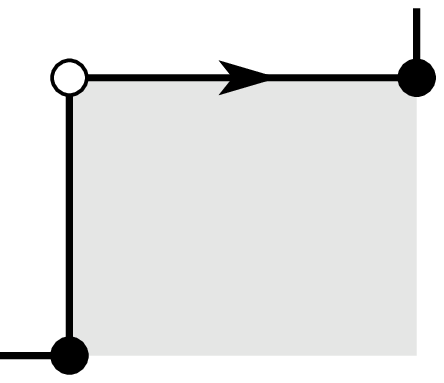}\put(-31,22){$r$}\\
\end{tabular}
\caption{The change of~$\Gamma_{\protect\overrightarrow{\mathrm{II}}}(R)$ under allowed elementary moves on~$R$}\label{allowed-fig}
\end{figure}
Thus, the homology class of~$\Gamma$ in~$H_1(\mathbb T^2)$
is the same as that of~$\Gamma'$. 
Let~$(k,l)\in H_1(\mathbb T^2;\mathbb Z)$ be this class.

Since both multi-valued functions~$\theta$ and~$\varphi$ are locally non-decreasing on every edge of~$\Gamma$ and~$\Gamma'$,
each meridian~$\{\theta\}\times\mathbb S^1$ (respectively, longitude~$\mathbb S^1\times\{\varphi\}$) not passing through a vertex of~$R$
intersects each of~$\Gamma$ and~$\Gamma'$ exactly~$k$ (respectively, $l$) times.

Let~$\theta_1,\theta_2,\varphi_1,\varphi_2$ be as in Definition~\ref{moves-def}. Denote the rectangle~$[\theta_1;\theta_2]\times[\varphi_1;\varphi_2]$
by~$r_0$. Three cases are possible:
\begin{itemize}
\item $r=r_0$,
\item $r=[\theta_2;\theta_1]\times[\varphi_1;\varphi_2]$, or
\item
$r=[\theta_1;\theta_2]\times[\varphi_2;\varphi_1]$.
\end{itemize}

By the assumption of Definition~\ref{moves-def}, the intersection of~$r_0$ with~$R$ is a subset of the set
of vertices of~$r_0$.
Therefore, any vertical edge of~$\Gamma$ that intersects~$(\theta_1;\theta_2)\times\{\varphi_1\}$
intersects also~$(\theta_1;\theta_2)\times\{\varphi_2\}$, and vice versa. Let~$k_0$ be the number
of such edges. These edges are the same in~$\Gamma'$.

Similarly, let~$l_0$ be the number of horizontal edges of~$\Gamma$ (equivalently, of~$\Gamma'$)
that intersect~$\{\theta_1\}\times(\varphi_1;\varphi_2)$ (equivalently, $\{\theta_2\}\times(\varphi_1;\varphi_2)$).

$\Gamma$ and~$\Gamma'$ have exactly the same set of double points outside~$\partial r$. From the arguments above
it follows that the number
of double points of~$\Gamma$ and~$\Gamma'$ at~$\partial r$ is also the same and is equal to
\begin{itemize}
\item
$k_0+l_0$ if~$r=r_0$,
\item
$k-k_0-1+l_0$ if~$r=[\theta_2;\theta_1]\times[\varphi_1;\varphi_2]$,
\item
$k_0+l-l_0-1$ if~$r=[\theta_1;\theta_2]\times[\varphi_2;\varphi_1]$.
\end{itemize}
The claim follows.
\end{proof}

\begin{lemm}\label{step-lem}
Let~$R$ and~$R'$ be rectangular diagrams of a link such that the closure of the symmetric
difference~$\Gamma_{\protect\overrightarrow{\mathrm{II}}}(R)\triangle\Gamma_{\protect\overrightarrow{\mathrm{II}}}(R')$
has the form of the boundary of an embedded disk or an annulus~$F\subset\mathbb T^2$ such that
the interior of~$F$ is disjoint from~$R\cup R'$, and~$\partial F$ is disjoint from the set of double points of~$\Gamma_{\protect\overrightarrow{\mathrm{II}}}(R)$
and~$\Gamma_{\protect\overrightarrow{\mathrm{II}}}(R')$.
Then~$[R]_{\overrightarrow{\mathrm{II}}}=[R']_{\overrightarrow{\mathrm{II}}}$.
\end{lemm}

\begin{proof}
We again put~$\Gamma=\Gamma_{\overrightarrow{\mathrm{II}}}(R)$ and~$\Gamma'=\Gamma_{\overrightarrow{\mathrm{II}}}(R')$.
Suppose that~$F$ is a disc. It follows from the hypothesis of the lemma that:
\begin{enumerate}
\item
$F$ is co-bounded by two staircase arcs~$\alpha$ and~$\beta$ such that~$\alpha\subset\Gamma$
and~$\beta\subset\Gamma'$ (on which the functions~$\theta$ and~$\varphi$
are locally non-decreasing);
\item
the set of corners of~$F$ coincides with~$R\triangle R'$.
\end{enumerate}
See Figure~\ref{disc-d-fig}(a) for an illustration.
\begin{figure}[ht]
\begin{tabular}{ccc}
\includegraphics{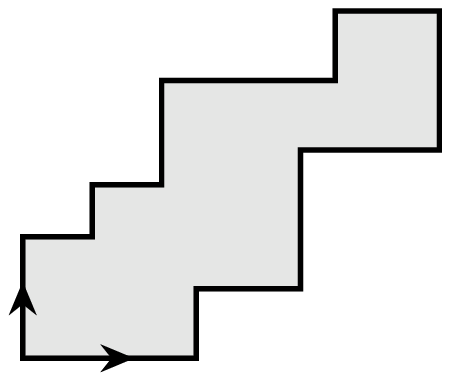}\put(-100,70){$\alpha$}\put(-47,40){$\beta$}\put(-72,55){$F$}&&
\includegraphics{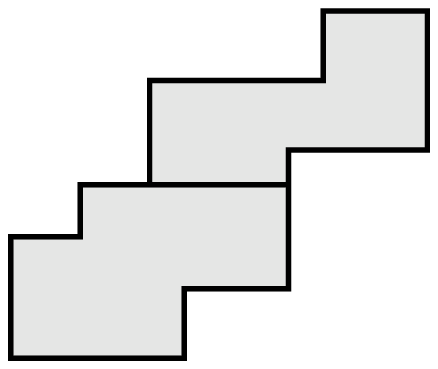}\put(-72,52){$\gamma$}\put(-100,30){$F_2$}\put(-50,78){$F_1$}\\
(a)&\hbox to 2cm{\hss}&(b)
\end{tabular}
\caption{Induction step when~$F$ is a disc}\label{disc-d-fig}
\end{figure}

The proof is by induction
in the number~$m$ of corners of the polygon~$F$.
The smallest possible number is~$m=4$. In this case, one easily finds
that~$R\mapsto R'$ is an exchange move or a type~$\overrightarrow{\mathrm{II}}$ stabilization or destabilization.

Suppose that~$m>4$ and the claim is proved in the case when~$F$ has fewer corners than~$m$.
Small perturbations of a rectangular diagram of a link are achievable by means exchange moves, so, without loss of generality, we may assume
that no meridian or longitude of the torus~$\mathbb T^2$ contains four points of~$R\cup R'$, for this can
be resolved by a small perturbation of~$R$ or~$R'$.

There is an arc~$\gamma$ of the form~$[\theta_1;\theta_2]\times\{\varphi_0\}$ such that:
\begin{enumerate}
\item
$\gamma\subset F$;
\item
$\gamma\cap\partial F=\partial\gamma$;
\item
one of the endpoints of~$\gamma$ belongs to~$R\cup R'$.
\end{enumerate}

Without loss of generality we may assume that~$(\theta_1,\varphi_0)\in\alpha$ and~$(\theta_2,\varphi_0)\in\beta$ as this is
the question of exchanging the roles of~$R$ and~$R'$. The arc~$\gamma$ cuts~$F$ into two discs, which we denote by~$F_1$ and~$F_2$.
We number them so that~$F_1$ is above~$\gamma$ and~$F_2$ is below~$\gamma$, see Figure~\ref{disc-d-fig}(b).

Let~$C_1$ (respectively, $C_2$) be the set of corners of~$F_1$ (respectively, $F_2$). One can see that there is
a rectangular diagram of a link~$R''$ such that~$R\triangle R''=C_1$ (which is equivalent to~$R'\triangle R''=C_2$)
whose orientation agrees with that of~$R$ on~$R\cap R''$. We then have
\begin{equation}\label{f1f2-eq}\overline{\Gamma\triangle\Gamma_{\protect\overrightarrow{\mathrm{II}}}(R'')}=\partial F_1,\quad
\overline{\Gamma'\triangle\Gamma_{\protect\overrightarrow{\mathrm{II}}}(R'')}=\partial F_2.
\end{equation}
Each of~$F_1$ and~$F_2$ has fewer corners than~$m$, hence, by the induction hypothesis, we have
$[R]_{\overrightarrow{\mathrm{II}}}=[R'']_{\overrightarrow{\mathrm{II}}}$ and~$[R']_{\overrightarrow{\mathrm{II}}}=[R'']_{\overrightarrow{\mathrm{II}}}$.
The induction step follows.

Now suppose that~$F$ is an annulus. Then it can be cut by two straight line segments, one horizontal and one vertical,
into two discs so that condition~\eqref{f1f2-eq} will hold (possibly after exchanging~$R$ and~$R'$),
which again will imply $[R]_{\overrightarrow{\mathrm{II}}}=[R'']_{\overrightarrow{\mathrm{II}}}$
and~$[R']_{\overrightarrow{\mathrm{II}}}=[R'']_{\overrightarrow{\mathrm{II}}}$
by the proven case of the lemma. The idea is illustrated in Figure~\ref{annf-fig}. We skip the easy details.
\begin{figure}[ht]
\centerline{\includegraphics{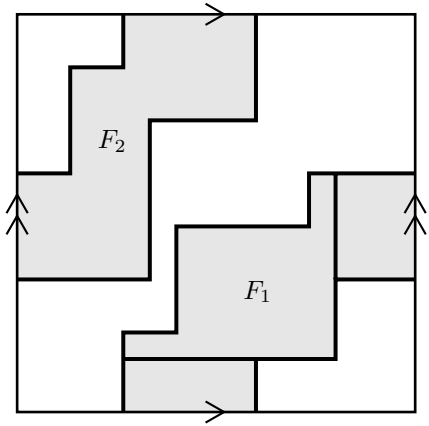}\put(-75,53){$F_1$}\put(-130,110){$F_2$}}
\caption{Cutting the annulus~$F$ into two discs~$F_1$ and~$F_2$}\label{annf-fig}
\end{figure}
\end{proof}

\begin{lemm}\label{r=r-lem}
Let~$R$ and~$R'$ be rectangular diagrams of a link such that:
\begin{enumerate}
\item
$\Gamma_{\overrightarrow{\mathrm{II}}}(R)$ and~$\Gamma_{\overrightarrow{\mathrm{II}}}(R')$
have the same set of double points, which we denote by~$X$;
\item
there is an isotopy from~$\Gamma_{\overrightarrow{\mathrm{II}}}(R)$ to~$\Gamma_{\overrightarrow{\mathrm{II}}}(R')$ in~$\mathbb T^2$
fixed on an open neighborhood of~$X$.
\end{enumerate}
Then~$[R]_{\overrightarrow{\mathrm{II}}}=[R']_{\overrightarrow{\mathrm{II}}}$.
\end{lemm}

\begin{proof}
As before, we put~$\Gamma=\Gamma_{\overrightarrow{\mathrm{II}}}(R)$ and~$\Gamma'=\Gamma_{\overrightarrow{\mathrm{II}}}(R')$
and assume that no longitude or meridian contains four points of~$R\cup R'$.

The closure~$\alpha$
of a connected component of~$\Gamma\setminus X$ (respectively, $\Gamma'\setminus X$) will be called \emph{an arc of~$\Gamma$} (respectively, $\Gamma'$) if~$\alpha\cap X\ne\varnothing$.
There is a natural one-to-one correspondence between the arcs of~$\Gamma$ and those of~$\Gamma'$, defined
by demanding that arcs~$\alpha\subset\Gamma$ and~$\alpha'\subset\Gamma'$ corresponding to each other
have the same starting (equivalently, terminal) portion.
Let~$\alpha_1,\alpha_2,\ldots,\alpha_N$ be all the arcs of~$\Gamma$, and let~$\alpha_1',\alpha_2',\ldots,\alpha_N'$ be the
respective arcs of~$\Gamma'$.

Some connected components of~$\Gamma$ (and hence of~$\Gamma'$) may be disjoint from~$X$, and thus be simple
closed staircase-like curves, which are pairwise disjoint and homologous to each other.
Let~$\gamma_1,\gamma_2,\ldots,\gamma_K$ be all these components numbered using the following recipe.
Choose a point~$x_0\in X$ if $X$ is non-empty, and~$x_0\in\mathbb T^2\setminus(\Gamma\cup\Gamma')$ otherwise.
Choose also an oriented loop~$\beta\subset\mathbb T^2$ based at~$x_0$ that intersects each~$\gamma_i$ exactly once.
The numbering of~$\gamma_i$'s is chosen according to the order in which the points~$\gamma_i\cap\beta$ follow on~$\beta$.

Let~$\gamma_1',\gamma_2',\ldots,\gamma_K'$ be the closed components of~$\Gamma'\setminus X$
numbered so that an isotopy bringing~$\Gamma$ to~$\Gamma'$ and fixed on~$X\cup\{x_0\}$ brings~$\gamma_i$ to~$\gamma_i'$, $i=1,\ldots,K$.

We proceed by induction in
\begin{equation}\label{cRR-eq}
c(R,R')=(K+1)\Bigl(\sum_{\beta,\beta'}\chi(\beta\cap\beta')-N\Bigr)+\bigl|\{i=1,\ldots,K:\gamma_i\ne\gamma_i'\}\bigr|,
\end{equation}
where~$\chi$ denotes the Euler characteristics, and the sum is taken over all connected components~$\beta$ of~$\Gamma\setminus X$
and all connected components~$\beta'$ of~$\Gamma'\setminus X$.

The equality~$c(R,R')=0$ means that~$R$ and~$R'$ coincide. This is the induction base.

Suppose that~$c(R,R')>0$ and~$\sum_{\beta,\beta'}\chi(\beta\cap\beta')=N$. This means that all the arcs of~$\Gamma'$
coincide with the respective arcs of~$\Gamma$, and, for any~$i,j\in\{1,\ldots,K\}$, the curves~$\gamma_i$ and~$\gamma_j'$
are either coincident or disjoint.

Let~$k$ be the minimal index such that~$\gamma_k\ne\gamma_k'$. Then~$\gamma_k$ and~$\gamma_k'$ cut
the torus~$\mathbb T^2$ into two annuli. Let~$A$ be the one of these annuli that does not contain the point~$x_0$.
The interior of~$A$ is disjoint either from~$\Gamma$ or from~$\Gamma'$. Without loss of generality
we may assume the former. There is a rectangular diagram of a link~$R''$ such that~$\Gamma_{\overrightarrow{\mathrm{II}}}(R'')=
(\Gamma\setminus\gamma_k)\cup\gamma_k'$. We have~$[R]_{\overrightarrow{\mathrm{II}}}=[R'']_{\overrightarrow{\mathrm{II}}}$
by Lemma~\ref{step-lem} and~$c(R'',R')<c(R,R')$, which gives the induction step.

Now suppose that~$\sum_{\beta,\beta'}\chi(\beta\cap\beta')>N$. This means that, for some~$i\in\{1,\ldots,N\}$,
we have~$\alpha_i\ne\alpha_i'$ or, for some~$i,j\in\{1,\ldots,K\}$, we have~$\gamma_i\cap\gamma_j'\ne\varnothing$,
$\gamma_i\ne\gamma_j'$. In both cases, we claim that either~$c(R,R')$ can be reduced by
a small perturbation of~$R$ or~$R'$ keeping the set of double points of~$\Gamma$ fixed, or
there is a disc~$d\subset\mathbb T^2$ co-bounded by two staircase arcs~$\beta,\beta'$ such
that
$$d\cap\Gamma=\beta\subset\Gamma\setminus X,\qquad d\cap\Gamma'=\beta'\subset\Gamma'\setminus X.$$

Take this claim for granted for the moment. In the former case, the induction step is obvious. In the latter case,
there is a rectangular diagram of a link~$R''$ such that~$\Gamma_{\overrightarrow{\mathrm{II}}}(R'')=(\Gamma\setminus\beta)\cup
\beta'$. By Lemma~\ref{step-lem}, for such a diagram, we again have~$[R]_{\overrightarrow{\mathrm{II}}}=[R'']_{\overrightarrow{\mathrm{II}}}$.
The first summand in~\eqref{cRR-eq} decreases by~$K+1$ when~$R$ is replaced by~$R''$,
whereas the second summand may increase by at most~$K$
(as a result of possible renumbering of~$\gamma_i$'s). Hence~$c(R'',R')<c(R,R')$, and the induction step follows.

Now we prove the claim.
A disc in~$\mathbb T^2$ disjoint from~$X$ and
co-bounded by a subarc of~$\Gamma\setminus X$ and a subarc of~$\Gamma'\setminus X$ will be referred to as \emph{a bigon of~$\Gamma$
and~$\Gamma'$}. If these subarcs are the \emph{only} intersections of the bigon with~$\Gamma$ and~$\Gamma'$, then the bigon will be called \emph{clean}. We use a similar terminology for the full preimages~$\widetilde\Gamma$ and~$\widetilde\Gamma'$
of~$\Gamma$ and~$\Gamma'$, respectively,
under the projection map~$\mathbb R^2\rightarrow\mathbb T^2=\mathbb R^2/(2\pi\mathbb Z^2)$.
The set of double points of~$\widetilde\Gamma$ (equivalently, of~$\widetilde\Gamma'$) is denoted by~$\widetilde X$.

Suppose that~$\alpha_i\ne\alpha'_i$ for some~$i\in\{1,\ldots,N\}$. Choose preimages~$\widetilde\alpha_i$ and~$\widetilde\alpha_i'$
of these arcs
in~$\mathbb R^2$ so that~$\partial\widetilde\alpha_i=\partial\widetilde\alpha_i'$.
(The arcs~$\alpha_i$ and~$\alpha_i'$ may form closed loops based at a point from~$X$,
in which case~$\widetilde\alpha_i$ and~$\widetilde\alpha_i'$ are defined as the closures
of preimages of~$\alpha_i\setminus X$ and~$\alpha_i'\setminus X$, such that~$\partial\widetilde\alpha_i=\partial\widetilde\alpha_i'$.)

By the hypothesis of the lemma, the staircase
arcs~$\widetilde\alpha_i$ and~$\widetilde\alpha'_i$ are isotopic
relative to~$\widetilde X$ and coincide near~$\partial\widetilde\alpha_i=\partial\widetilde\alpha_i'\subset\widetilde X$.
This implies the existence of a bigon~$\widetilde d$ of~$\widetilde\Gamma$ and~$\widetilde\Gamma'$
with~$\partial\widetilde d\subset\widetilde\alpha_i\cup\widetilde\alpha_i'$.
However, this bigon is not necessarily clean. If the interior of~$\widetilde d$
has a non-empty intersection with~$\widetilde\Gamma$ or~$\widetilde\Gamma'$,
then a subarc of~$\widetilde\Gamma\setminus\widetilde X$ or~$\widetilde\Gamma'\setminus\widetilde X$
cuts off a smaller bigon from~$\widetilde d$. Let~$\widetilde d_0$ be a minimal bigon of~$\widetilde\Gamma$ and~$\widetilde\Gamma'$
contained in~$\widetilde d$, that is, such that there is no smaller bigon contained in~$\widetilde d_0$.

Let~$\widetilde\beta\subset\widetilde\Gamma\setminus\widetilde X$ and~$\widetilde\beta'\subset\widetilde\Gamma'\setminus\widetilde X$
be the arcs co-bounding~$\widetilde d_0$. By construction, the interior of~$\widetilde d_0$ is disjoint from~$\widetilde\Gamma$
and~$\widetilde\Gamma'$.
If~$\widetilde\beta$ has a non-empty intersection with~$\widetilde\Gamma'$, or~$\widetilde\beta'$ has a non-empty
intersection with~$\widetilde\Gamma$,
then this intersection can be resolved by a small perturbation of~$R$ or~$R'$, which results in decreasing of~$c(R,R')$.
If~$\beta\cap\Gamma'=\varnothing=\beta'\cap\Gamma$, then the bigon~$\widetilde d_0$ is clean, and so is its image~$d_0$
in~$\mathbb T^2$. The claim follows.

Now suppose that~$\alpha_i=\alpha_i'$ for all~$i=1,\ldots,N$.
Let $i,j\in\{1,\ldots,K\}$ be such that~$\gamma_i\cap\gamma_j'\ne\varnothing$, $\gamma_i\ne\gamma_j'$.
If~$\gamma_i\cap\gamma_j'$ is a single point, this intersection can be resolved by a small perturbation of~$R$ or~$R'$,
which results in decreasing of~$c(R,R')$.
Otherwise we find a bigon~$\widetilde d$
of~$\widetilde\Gamma$ and~$\widetilde\Gamma'$ co-bounded by some~$\beta\subset\widetilde\gamma_i$
and~$\beta'\subset\widetilde\gamma_j'$, where~$\widetilde\gamma_i$ and~$\widetilde\gamma_j'$ are preimages of~$\gamma_i$
and~$\gamma_j'$, respectively, in~$\mathbb R^2$, and proceed as above.
\end{proof}

\begin{proof}[Proof of Theorem~\ref{main-tech-theo}]
Due to symmetry it suffices to prove the assertion for any of the four types of stabilizations. We choose~$T=\overrightarrow{\mathrm{II}}$.

Two rectangular diagrams of a link~$R$ and~$R'$ (or, more generally, any two pairs~$(R^+,R^-)$, $({R'}^+,{R'}^-)$
of finite subsets of~$\mathbb T^2$)
are called \emph{combinatorially equivalent} if
there are orientation-preserving self-homeomorphisms~$f,g$ of~$\mathbb S^1$ such that~$(f\times g)(R^\pm)={R'}^\pm$.

Two rectangular diagrams of a link~$R$ and~$R'$
are said to be of the same \emph{$\overrightarrow{\mathrm{II}}$-homology type} if there is a rectangular diagram of a link~$R''$
such that~$R'$ and~$R''$ are combinatorially equivalent, and~$\Gamma=\Gamma_{\overrightarrow{\mathrm{II}}}(R)$
is isotopic to~$\Gamma_{\overrightarrow{\mathrm{II}}}(R'')$ relative to the set of double points of~$\Gamma$.
This is clearly an equivalence relation.
It follows from Lemma~\ref{r=r-lem} that the coincidence of the $\overrightarrow{\mathrm{II}}$-homology types
of~$R$ and~$R'$ implies~$[R]_{\overrightarrow{\mathrm{II}}}=[R']_{\overrightarrow{\mathrm{II}}}$.

\begin{rema}
The term `homology type' is justified by the fact that the $\overrightarrow{\mathrm{II}}$-homology type of a diagram~$R$ is
determined by the homological information about~$\Gamma_{\overrightarrow{\mathrm{II}}}(R)$, which
can be encoded by the function~$\psi:H_1(\Gamma,X;\mathbb Z)\rightarrow\mathbb Z$ defined by
$$\psi(z)=\bigl|\{\alpha\subset\Gamma\setminus X:[\overline\alpha]=z\}\bigr|,$$
where~$\Gamma=\Gamma_{\overrightarrow{\mathrm{II}}}(R)$, and~$X$ is the set of double points of~$\Gamma$.
\end{rema}

Now we claim that, for any~$k,l\in\mathbb N$ and~$m\in\mathbb N\cup\{0\}$, there are only finitely many pairwise
distinct $\overrightarrow{\mathrm{II}}$-homology types of diagrams~$R$ such that~$\omega_{\overrightarrow{\mathrm{II}}}(R)=(k,l,m)$.

Indeed, if~$\omega_{\overrightarrow{\mathrm{II}}}(R)=(k,l,0)$, then~$\Gamma_{\overrightarrow{\mathrm{II}}}(R)$ is a union
of~$\mathrm{lcd}(k,l)$ simple closed curves in~$\mathbb T^2$ having homology class~$(k,l)/\mathrm{lcd}(k,l)$.
This means that the $\overrightarrow{\mathrm{II}}$-homology type is completely determined by~$k,l$.

Suppose $\omega_{\overrightarrow{\mathrm{II}}}(R)=(k,l,m)$ with~$m>0$. Denote by~$X$ the set of double points
of~$\Gamma=\Gamma_{\overrightarrow{\mathrm{II}}}(R)$. Let~$\theta_1,\theta_2,\ldots,\theta_K\in\mathbb S^1$
(respectively, $\varphi_1,\varphi_2,\ldots,\varphi_L\in\mathbb S^1$) be all the points in the projection~$p_\theta(X)$ (respectively,
$p_\varphi(X)$) numbered according to their cyclic order in~$\mathbb S^1$. We have~$K,L\leqslant m$.

Pick an~$\varepsilon>0$ smaller than one half of the length of the shortest interval among those into which~$\mathbb S^1$
is cut by~$p_\theta(R)\cup p_\varphi(R)$. Then whenever~$(\theta_i,\varphi_j)\in X$, we will have
\begin{equation}\label{cross-eq}
\Gamma\cap[\theta_i-\varepsilon;\theta_i+\varepsilon]\times[\varphi_j-\varepsilon;\varphi_j+\varepsilon]=
\bigl([\theta_i-\varepsilon;\theta_i+\varepsilon]\times\{\varphi_j\}\bigr)\cup
\bigl(\{\theta_i\}\times[\varphi_j-\varepsilon;\varphi_j+\varepsilon]\bigr).
\end{equation}

Denote the set~$\bigcup_{i=1}^K\{\theta_i-\varepsilon,\theta_i+\varepsilon\}\subset\mathbb S^1$ by~$\Theta$
and~$\bigcup_{j=1}^L\{\varphi_j-\varepsilon,\varphi_j+\varepsilon\}\subset\mathbb S^1$ by~$\Phi$.
Due to the choice of~$\varepsilon$, we have~$\Theta\cap p_\theta(R)=\varnothing=\Phi\cap p_\varphi(R)$
since~$p_\theta(X)\subset p_\theta(R)$ and~$p_\varphi(X)\subset p_\varphi(R)$.
Therefore, whenever~$\theta_0\in\Theta$ (respectively, $\varphi_0\in\Phi$), the meridian~$\{\theta_0\}\times\mathbb S^1$
(respectively, the longitude~$\mathbb S^1\times\{\varphi_0\}$) intersects~$\Gamma$ in exactly~$k$ (respectively,~$l$)
points.

Denote by~$Y$ the set of all such intersection points:
$$Y=\bigl((\Theta\times\mathbb S^1)\cup(\mathbb S^1\times\Phi)\bigr)\cap\Gamma.$$
We claim that the homology type of~$R$ can be recovered from~$X$ and~$Y$. Indeed, we can recover
the subsets~$\Theta$ and~$\Phi$ as they are the projections~$p_\theta(Y)$ and~$p_\varphi(Y)$.

Now let~$r$ be the closure of a connected component of~$\mathbb T^2\setminus\bigl((\Theta\times\mathbb S^1)\cup(\mathbb S^1\times\Phi)\bigr)$.
By construction, $r$ is a rectangle which is either disjoint from~$X$ or contains exactly one point from~$X$.
In the former case, we can recover~$\Gamma\cap r$ up to isotopy relative to~$\partial r$
since~$\partial r\cap\Gamma\subset Y$
and~$\Gamma\cap r$ is a union of pairwise disjoint staircase arcs on which the functions~$\theta$, $\varphi$
are non-decreasing (some of these arcs may be degenerate to a single point). In the latter case,
the intersection~$\Gamma\cap r$ is completely known due to~\eqref{cross-eq}.

The number of points in~$Y$ is bounded from above by a function of~$k,l,m$:
$$|Y|=2Kk+2Ll\leqslant2m(k+l).$$
Therefore, for any fixed triple~$(k,l,m)$ there are only finitely many combinatorial types of pairs~$(X,Y)$ that
can arise in this construction, and hence, the number of homology types or rectangular diagrams~$R$
with~$\omega_{\overrightarrow{\mathrm{II}}}(R)=(k,l,m)$ is also finite.

In a similar fashion one can show that, for any fixed triple~$(k,l,m)$,
the number of pairs of homology types~$(Z,Z')$ of rectangular diagrams
such that, for some~~$R\in Z$,
$R'\in Z'$, we have~$\omega_{\overrightarrow{\mathrm{II}}}(R)=\omega_{\overrightarrow{\mathrm{II}}}(R')=(k,l,m)$
and~$R\mapsto R'$
is either an exchange move or a type~$\overrightarrow{\mathrm{II}}$ stabilization, is also finite.

Thus, an algorithm to decide wether~$[R]_{\overrightarrow{\mathrm{II}}}=[R']_{\overrightarrow{\mathrm{II}}}$ is constructed as follows.
First, compute~$\omega_{\overrightarrow{\mathrm{II}}}(R)$ and~$\omega_{\overrightarrow{\mathrm{II}}}(R')$.
If~$\omega_{\overrightarrow{\mathrm{II}}}(R)\ne\omega_{\overrightarrow{\mathrm{II}}}(R')$,
then~$[R]_{\overrightarrow{\mathrm{II}}}\ne[R']_{\overrightarrow{\mathrm{II}}}$ by Lemma~\ref{omega-inv-lem}.

If~$\omega_{\overrightarrow{\mathrm{II}}}(R)=\omega_{\overrightarrow{\mathrm{II}}}(R')=(k,l,m)$ we construct a graph~$G$
whose vertices are homology types of all rectangular diagrams of links~$R''$ with~$\omega_{\overrightarrow{\mathrm{II}}}(R'')=(k,l,m)$,
and the edges are all pairs~$(Z_1,Z_2)$ of vertices such that there exists an exchange move or a type~$\overrightarrow{\mathrm{II}}$
stabilization~$R_1\mapsto R_2$ with~$R_1\in Z_1$, $R_2\in Z_2$. As we have seen above, this graph is finite.
It is also clear that a procedure to construct this graph as well as to find its vertices~$Z$, $Z'$
with~$Z\ni R$ and~$Z'\ni R'$ can be described in a purely combinatorial way.
Now the equality~$[R]_{\overrightarrow{\mathrm{II}}}=[R']_{\overrightarrow{\mathrm{II}}}$
holds if and only if the vertices~$Z$ and~$Z'$ belong to the same
connected component of~$G$, which is easily checkable.
\end{proof}

\begin{coro}\label{main-coro}
The equivalence problem for transverse links of a topological type that has trivial orientation-preserving
symmetry group is decidable.
\end{coro}

\begin{proof}
Recall that equivalence classes of positively $\xi_+$-transverse links can be viewed as $\xi_+$-Legendrian links modulo
Legendrian isotopy and negative stabilizations, and also as elements of~$\mathscr R/\bigl\langle\overrightarrow{\mathrm I},\overleftarrow{\mathrm I},\overrightarrow{\mathrm{II}}\bigr\rangle$, whereas equivalence classes of $\xi_+$-Legendrian (respectively, $\xi_-$-Legendrian) links are identified
with elements of~$\mathscr R/\langle\overrightarrow{\mathrm I},\overleftarrow{\mathrm I}\rangle$ (respectively,
$\mathscr R/\langle\overrightarrow{\mathrm{II}},\overleftarrow{\mathrm{II}}\rangle$) (see~\cite{dyn-shast,OST}).

The proof of the corollary is parallel to that of~\cite[Theorem~7.1]{dyn-shast}. Namely, for any two topologically equivalent
positively~$\xi_+$-transverse links we can find their presentations by rectangular diagrams~$R_1$, $R_2$
such that~$[R_1]_{\overrightarrow{\mathrm{II}},\overleftarrow{\mathrm{II}}}=[R_2]_{\overrightarrow{\mathrm{II}},\overleftarrow{\mathrm{II}}}$.
If the links~$\widehat R_1$ and~$\widehat R_2$ have trivial orientation-preserving symmetry group,
then by~\cite[Theorem~4.2]{dyn-shast}, we have
$$[R_1]_{\overrightarrow{\mathrm{II}}}=[R_2]_{\overrightarrow{\mathrm{II}}}\quad\Leftrightarrow\quad
[R_1]_{\overrightarrow{\mathrm I},\overleftarrow{\mathrm I},\overrightarrow{\mathrm{II}}}=
[R_2]_{\overrightarrow{\mathrm I},\overleftarrow{\mathrm I},\overrightarrow{\mathrm{II}}}.$$
An application of Theorem~\ref{main-tech-theo} completes the proof.
\end{proof}

\section{Transverse--Legendrian links}
For an introduction to contact topology and the theory of Legendrian and transverse knots
the reader is referred to~\cite{etnyre05} and~\cite{Ge}. Here we consider links which
are Legendrian with respect to one contact structure and transverse to another,
simultaneously.

Namely, let~$\xi_+$ and~$\xi_-$ be the cooriented contact structures on~$\mathbb S^3$ defined
by the $1$-forms
$$\alpha_+=\sin^2\Bigl(\frac{\pi\tau}2\Bigr)\,d\theta+\cos^2\Bigl(\frac{\pi\tau}2\Bigr)\,d\varphi\qquad\text{and}\qquad
\alpha_-=\sin^2\Bigl(\frac{\pi\tau}2\Bigr)\,d\theta-\cos^2\Bigl(\frac{\pi\tau}2\Bigr)\,d\varphi,$$
respectively, that is~$\xi_\pm=\ker\alpha_\pm$. One can see that~$\xi_+$ is nothing else but
the standard contact structure, and~$\xi_-$ is a mirror image of~$\xi_+$.

\begin{defi}
A smooth link in~$\mathbb S^3$ is called \emph{transverse-Legendrian of type~$\overrightarrow{\mathrm{II}}$}
(or simply \emph{transverse-Legendrian})
if it is positively transverse with respect to~$\xi_+$ and Legendrian with respect to~$\xi_-$.

Two transverse-Legendrian links are \emph{equivalent} if they are isotopic within the class of transverse-Legendrian links.
\end{defi}

Let~$L$ be a transverse-Legendrian link.
The contact structures~$\xi_+$ and~$\xi_-$ agree, if their coorientations are ignored, at the union~$\mathbb S^1_{\tau=0}\cup\mathbb
S^1_{\tau=1}$, since we have~$\alpha_+=\pm\alpha_-$ on this subset.
Therefore, $L$ misses the circles~$\mathbb S^1_{\tau=0}$ and~$\mathbb
S^1_{\tau=1}$, and the torus projection~$p_{\theta,\varphi}$ is well defined on the whole of~$L$.
One can also see that the restriction of both forms~$d\theta$ and~$d\varphi$ on~$L$ are non-degenerate and,
moreover, positive with respect to the orientation of~$L$.

Any transverse-Legendrian link~$L$ can be uniquely recovered from its torus projection similarly to the way in which
a Legendrian link is recovered from its front projection. Indeed, since~$\alpha_-|_L=0$,
the following equality holds for the restrictions of the coordinates~$\theta,\varphi,\tau$ on~$L$:
\begin{equation}\label{tau-eq}
\tau=\frac2\pi\sqrt{\arctan\frac{d\varphi}{d\theta}}.\end{equation}
This means that we can describe the set of transverse-Legendrian links completely in terms of torus projections.
Namely, the following statement holds:

\begin{prop}
The torus projection map gives rise to a one-to-one correspondence between transverse-Legendrian links
and subsets~$\Gamma\subset\mathbb T^2$ such that the following holds:
\begin{enumerate}
\item
$\Gamma$ is the image of a smooth immersion~$\mathbb S^1\sqcup\mathbb S^1\sqcup\ldots\sqcup\mathbb S^1\rightarrow\mathbb T^2$,
\item
the slope of~$\Gamma$ is everywhere positive, and
\item
$\Gamma$ has no self-tangencies.
\end{enumerate}
\end{prop}

A subset satisfying Conditions~1--3 of this proposition will be referred as \emph{a (positive) torus front}.

A torus front is said to be \emph{almost generic} if it has no self-intersections
of multiplicity higher than two, and \emph{generic} if, additionally, no meridian or longitude of~$\mathbb T^2$
contains more than one self-intersection points of the front. An example of a generic torus front is shown in Figure~\ref{t-front-ex-fig}.
\begin{figure}[ht]
\includegraphics[scale=.3]{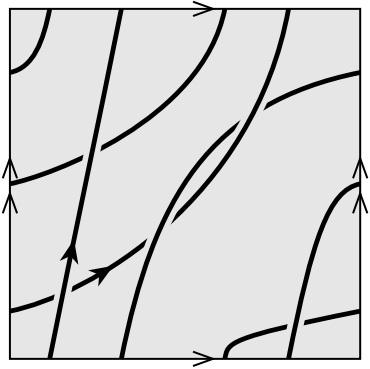}
\caption{A generic positive torus front}\label{t-front-ex-fig}
\end{figure}
We use a convention that, at every crossing point, the arc with with larger slope
is shown as overcrossing. Due to~\eqref{tau-eq},
this agrees with the position of the corresponding transverse-Legendrian link in~$\mathbb S^3$.
We also indicate the orientation of the corresponding transverse-Legendrian link.

\begin{prop}\label{R3-prop}
Two almost generic torus fronts define equivalent transverse-Legendrian links if and only if they are
obtained from one another by a sequence of continuous deformations in the class of almost generic torus fronts,
and type~III Reidemeister moves.
\end{prop}

\begin{proof}
This follows from the obvious fact that the main, codimension-one stratum of
the set of non-almost generic torus fronts consists of torus fronts having a triple self-intersection point.
\end{proof}

With every rectangular diagram of a link~$R$ we associate an equivalence class~$\mathrm{TL}_{\overrightarrow{\mathrm{II}}}(R)$
of transverse-Legendrian links
by demanding that a torus front representing an element of~$\mathrm{TL}_{\overrightarrow{\mathrm{II}}}(R)$
can be obtained by an arbitrarily small (in the $C^0$ sense) perturbation of~$\Gamma_{\overrightarrow{\mathrm{II}}}(R)$.
\begin{figure}[ht]
\begin{tabular}{ccc}
\includegraphics[width=150pt]{tl-example.eps}
&\hbox to 1cm{\hss}&
\includegraphics[width=150pt]{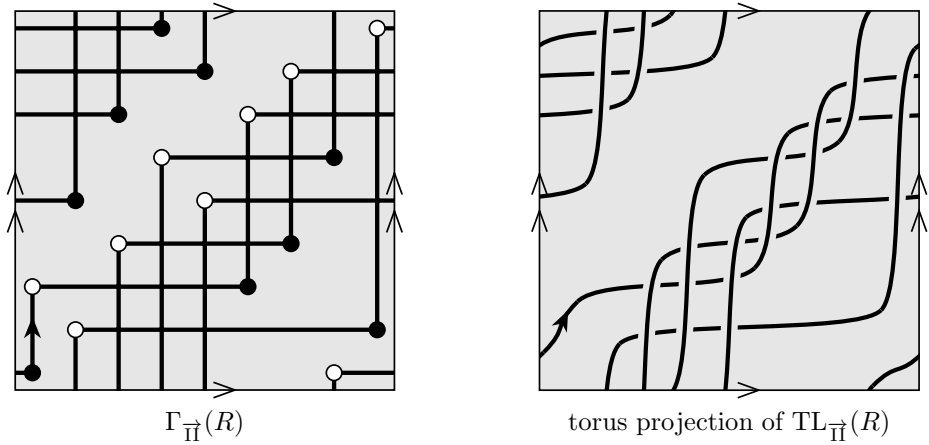}\\
$\Gamma_{\overrightarrow{\mathrm{II}}}(R)$&&
torus projection of~$\mathrm{TL}_{\overrightarrow{\mathrm{II}}}(R)$
\end{tabular}
\caption{Producing a transverse-Legendrian link from a rectangular diagram of a link}\label{tl-convertion-fig}
\end{figure}
See Figure~\ref{tl-convertion-fig} for an example.

\begin{prop}\label{i+ii-prop}
{\rm(i)}
Every equivalence class of type~$\overrightarrow{\mathrm{II}}$ transverse-Legendrian links has the form~$\mathrm{TL}_{\overrightarrow{\mathrm{II}}}(R)$
for some rectangular diagram of a link~$R$.

{\rm(ii)}
Let~$R$ and~$R'$ be rectangular diagrams of a link. Then~$\mathrm{TL}_{\overrightarrow{\mathrm{II}}}(R)=
\mathrm{TL}_{\overrightarrow{\mathrm{II}}}(R')$ implies~$[R]_{\overrightarrow{\mathrm{II}}}=[R']_{\overrightarrow{\mathrm{II}}}$.
\end{prop}

\begin{proof}
Statement~(i) follows from the approximation argument: any generic positive torus front~$\Gamma$ can be
approximated by a union of immersed staircase-like closed curves of the form~$\Gamma_{\overrightarrow{\mathrm{II}}}(R)$,
where~$R$ is a rectangular diagram of a surface, so that~$\Gamma_{\overrightarrow{\mathrm{II}}}(R)$ and~$\Gamma$
have the same set of double points.

To prove statement~(ii),
let~$\Gamma$ and~$\Gamma'$ be generic torus fronts obtained from~$\Gamma_{\overrightarrow{\mathrm{II}}}(R)$
and~$\Gamma_{\overrightarrow{\mathrm{II}}}(R')$, respectively, by a $C^0$-small perturbation.
The equality $\mathrm{TL}_{\overrightarrow{\mathrm{II}}}(R)=\mathrm{TL}_{\overrightarrow{\mathrm{II}}}(R')$
means that there is a continuous $1$-parametric family~$\Gamma_t$, $t\in[0;1]$ of torus fronts such that~$\Gamma_0=\Gamma$,
$\Gamma_1=\Gamma'$. Such a family can be chosen so that there are only finitely many moments~$t=t_1,t_2,\ldots,t_m$ at which the torus front~$\Gamma_t$
is not generic, and at these moments the genericity of~$\Gamma_t$ is unavoidably broken in one of the two simplest ways:
either there are two double points of~$\Gamma_t$ at the same meridian or longitude, or~$\Gamma_t$ has
a triple self-intersection. We may assume that~$t_1<t_2<\ldots<t_m$. We also set~$t_0=0$, $t_{m+1}=1$.

For each~$t\in[0;1]\setminus\{t_1,\ldots,t_m\}$ let~$R_t$ be a rectangular diagram of a link such that~$\Gamma_{\overrightarrow{\mathrm{II}}}(R_t)$
is isotopic to~$\Gamma_t$ relative the set of self-intersections of~$\Gamma_t$. By construction,
the homology type of~$R_t$ is constant
on each of the intervals~$[0;t_1)$, $(t_1;t_2),\ldots$, $(t_{m-1};t_m)$, $(t_m;1]$, and thus, by Lemma~\ref{r=r-lem}, so is~$[R_t]_{\overrightarrow{\mathrm{II}}}$.
At any critical moment~$t_i$ the torus front~$\Gamma_{t_i}$ can be approximated in two different ways
by unions of staircase curves~$\Gamma_{\overrightarrow{\mathrm{II}}}(R_{t_i}')$
and~$\Gamma_{\overrightarrow{\mathrm{II}}}(R_{t_i}'')$, where~$R_{t_i}'$ and~$R_{t_i}''$ are rectangular diagrams of a link such that:
\begin{enumerate}
\item
$R_{t_i}'^-\mapsto R_{t_i}''$ is an exchange move;
\item
the homology type of~$R_{t_i}'$ (respectively, $R_{t_i}''$)
coincides with that of~$R_t$ for~$t\in(t_{i-1};t_i)$ (respectively,
$t\in(t_i;t_{i+1})$).
\end{enumerate}
This is illustrated in Figure~\ref{R-+}.
\begin{figure}[ht]
\begin{tabular}{cccc}
\raisebox{4mm}{(a)}&
\includegraphics[scale=.65]{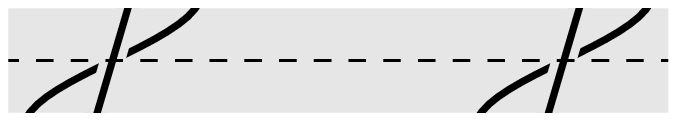}&
\includegraphics[scale=.65]{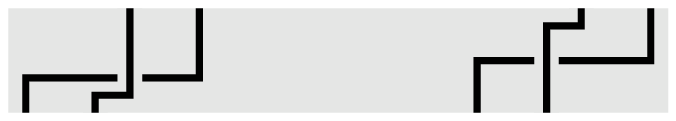}&
\includegraphics[scale=.65]{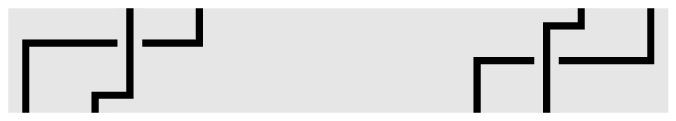}\\
&$\Gamma_{t_i}$&$\Gamma_{\overrightarrow{\mathrm{II}}}(R_{t_i}')$&$\Gamma_{\overrightarrow{\mathrm{II}}}(R_{t_i}'')$\\[5mm]
\raisebox{4mm}{(b)}&
\includegraphics[scale=.65]{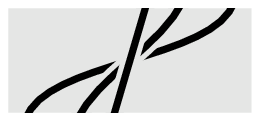}&
\includegraphics[scale=.65]{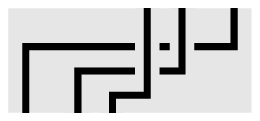}&
\includegraphics[scale=.65]{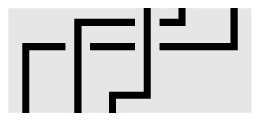}\\
&$\Gamma_{t_i}$&$\Gamma_{\overrightarrow{\mathrm{II}}}(R_{t_i}')$&$\Gamma_{\overrightarrow{\mathrm{II}}}(R_{t_i}'')$
\end{tabular}
\caption{Approximating a non-generic torus front: (a) when two double points appear on the same longitude; (b) when
a triple self-intersection occurs}\label{R-+}
\end{figure}

The claim follows.
\end{proof}

The converse to the assertion~(ii) of Proposition~\ref{i+ii-prop} does not hold in general. Namely,
the equality~$[R]_{\overrightarrow{\mathrm{II}}}=[R']_{\overrightarrow{\mathrm{II}}}$
does not always imply
$\mathrm{TL}_{\overrightarrow{\mathrm{II}}}(R)=
\mathrm{TL}_{\overrightarrow{\mathrm{II}}}(R')$.
However, the elements of~$\mathscr R/\bigl\langle\overrightarrow{\mathrm{II}}\bigr\rangle$
can be classified in terms of transverse-Legendrian links and exchange moves, which we now define.

\begin{defi}
Let~$\Gamma$ and~$\Gamma'$ be two generic positive torus fronts such that there are three
smooth arcs~$\alpha\subset\Gamma$, $\alpha'\subset\Gamma'$, and~$\beta\subset\Gamma\cap\Gamma'$
satisfying the following conditions (see Figure~\ref{3-arcs-fig}):
\begin{enumerate}
\item
the closure of the symmetric difference~$\Gamma\triangle\Gamma'$ is~$\alpha\cup\alpha'$;
\item
there is an embedded closed $2$-disc~$d\subset\mathbb T^2$ such~$\partial d=\alpha\cup\alpha'$;
\item
$\partial\beta\subset\partial d$;
\item $\int_\beta d\theta<2\pi$, $\int_\beta d\varphi<2\pi$;
\item
$\beta$ is homologous, relative to~$d$, either to a meridian or to a longitude of~$\mathbb T^2$;
\item
the intersection~$\beta\cap d$ consists of two arcs~$\gamma$ and~$\gamma'$ such
that~$\partial\gamma\subset\alpha$, $\partial\gamma'\subset\alpha'$;
\item
the interior of~$d$ intersects~$\Gamma\setminus\beta$ (equivalently, $\Gamma'\setminus\beta$) in a union of
pairwise disjoint open arcs each of which separates~$\gamma\setminus\partial\gamma$ from~$\gamma'\setminus\partial\gamma'$;
\item
if $\beta$ is homologous to a meridian (respectively, longitude) relative to~$d$, then at each self-intersection point of~$\Gamma$
or~$\Gamma'$ located at $\partial d\setminus\beta$
the overpassing (respectively, underpassing) arc is a part of~$\partial d$.
\end{enumerate}
Then we say that the passage from~$\Gamma$ to~$\Gamma'$ (or between the respective transverse-Legendrian links)
is called \emph{an exchange move}.
\begin{figure}[ht]
\centerline{\includegraphics[scale=.5]{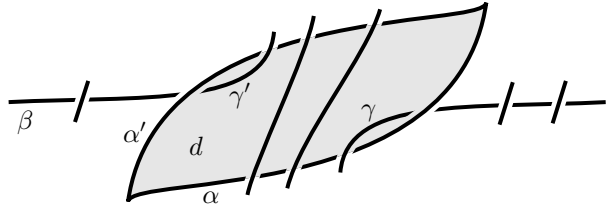}\put(-160,20){$d$}\put(-185,25){$\alpha'$}\put(-155,2){$\alpha$}
\put(-225,30){$\beta$}\put(-145,40){$\gamma'$}\put(-95,35){$\gamma$}}
\caption{The disc~$d$ and the arcs~$\alpha,\beta$ in the definition of an exchange move of
transverse-Legendrian links}\label{3-arcs-fig}
\end{figure}
\end{defi}

Exchange moves of transverse-Legendrian links are illustrated in Figure~\ref{big-mov-fig}.
\begin{figure}[ht]
\centerline{\includegraphics[scale=.3]{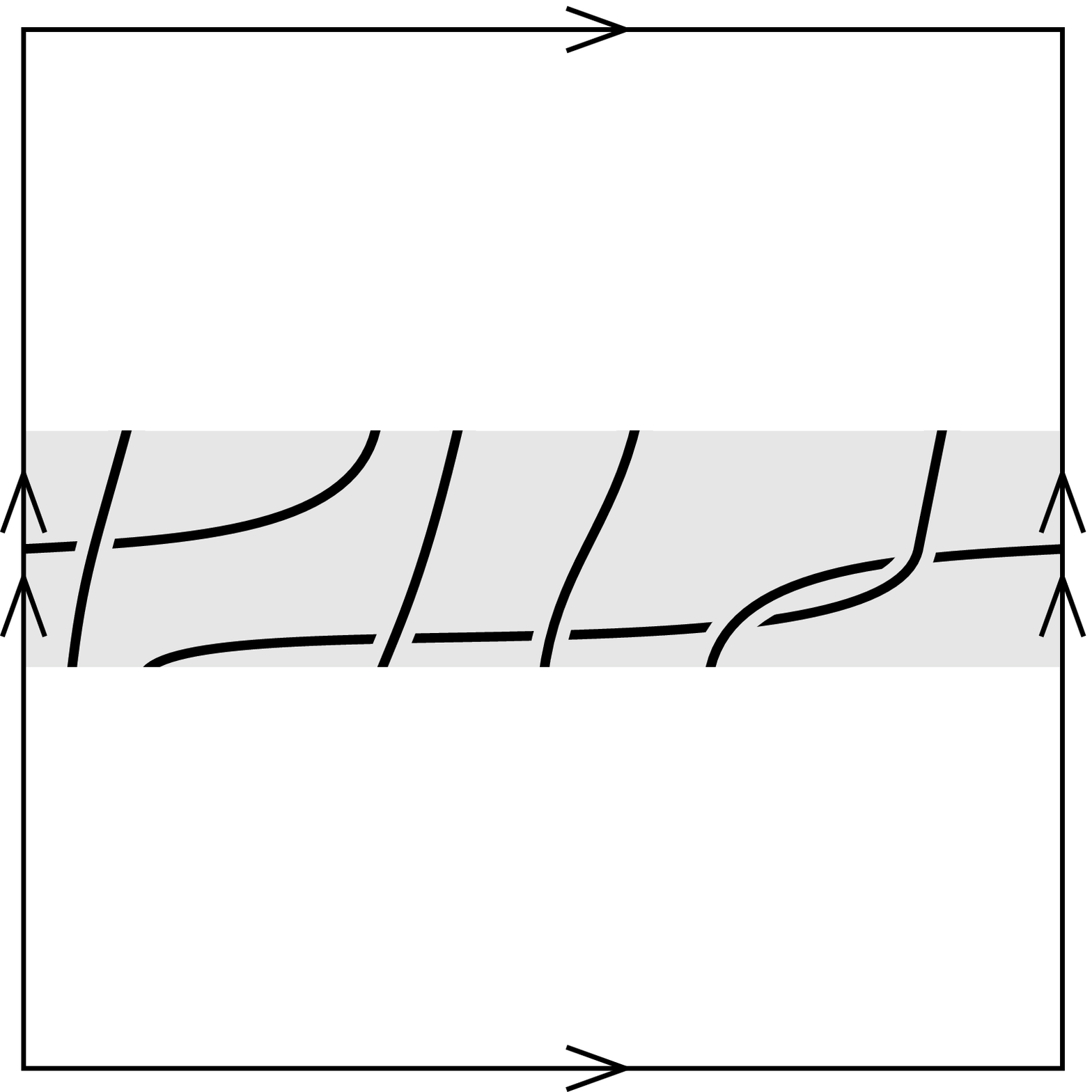}
\hskip1cm\raisebox{65pt}{$\longleftrightarrow$}\hskip1cm
\includegraphics[scale=.3]{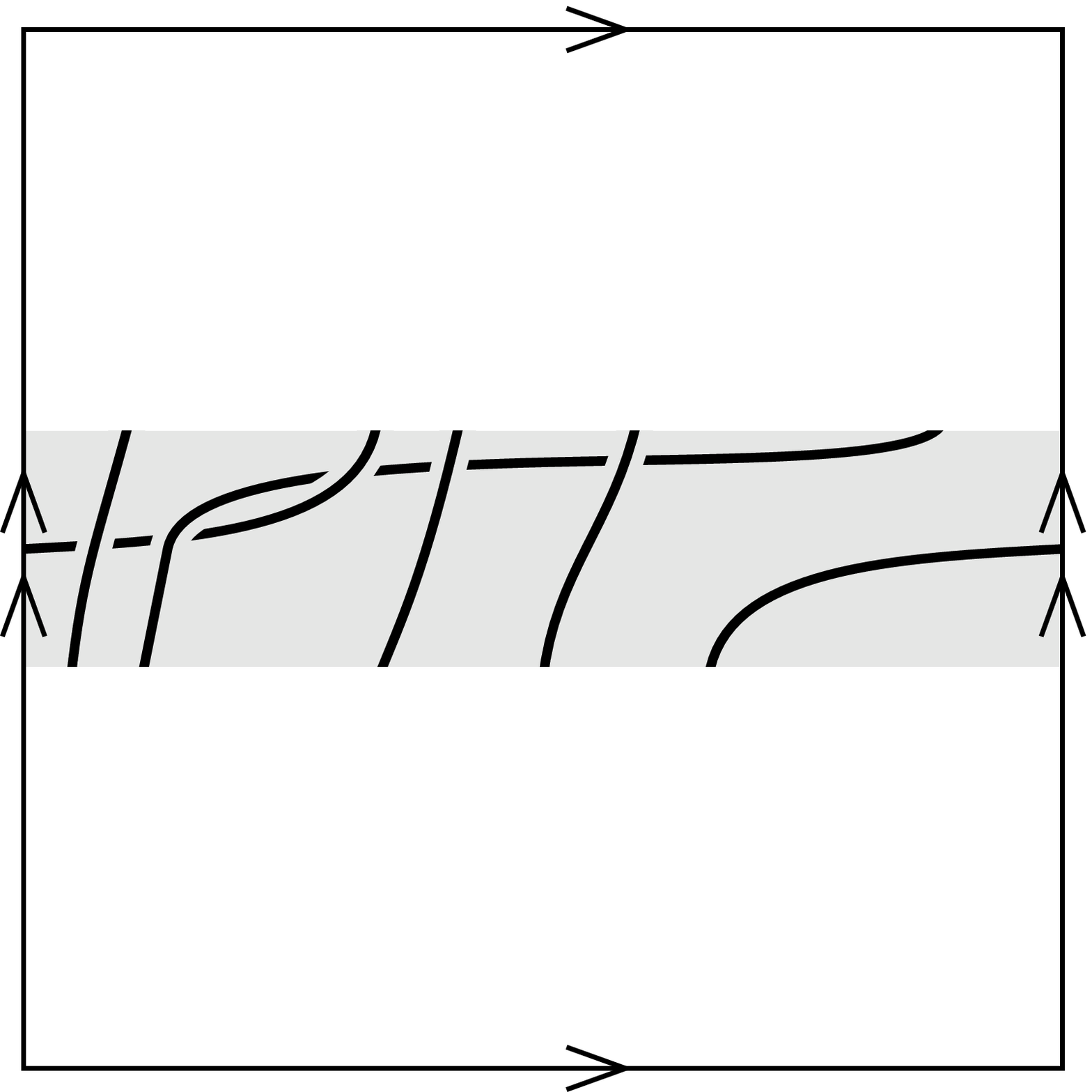}}

\vskip1cm
\centerline{\includegraphics[scale=.3]{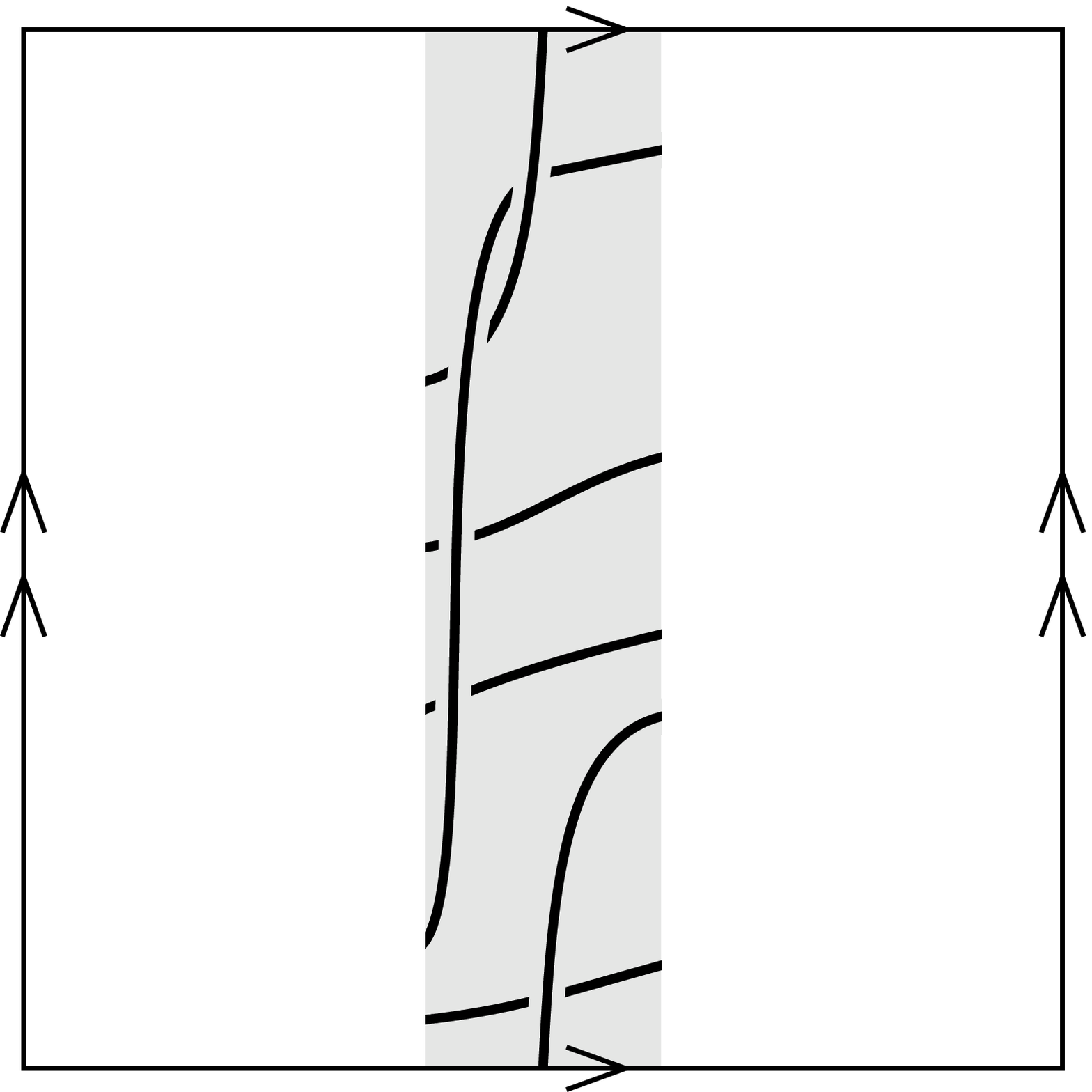}
\hskip1cm\raisebox{65pt}{$\longleftrightarrow$}\hskip1cm
\includegraphics[scale=.3]{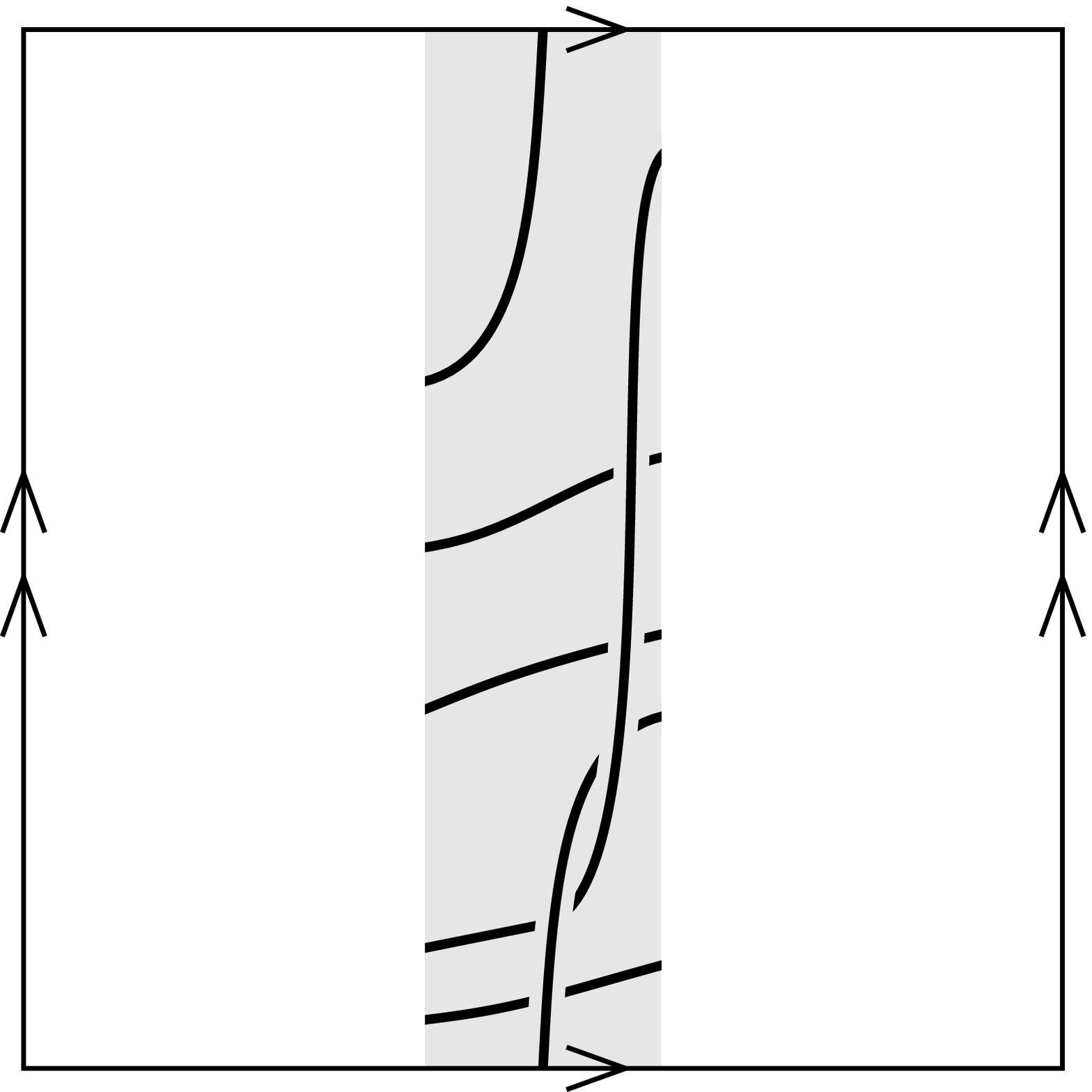}}
\caption{Exchange moves of transverse-Legendrian links}\label{big-mov-fig}
\end{figure}

\begin{prop}\label{class-prop}
Let~$R$ and~$R'$ be rectangular diagrams of a link. Then we have~$[R]_{\overrightarrow{\mathrm{II}}}=[R']_{\overrightarrow{\mathrm{II}}}$
if and only if the type~$\overrightarrow{\mathrm{II}}$ transverse-Legendrian link associated with~$R$ and~$R'$
can be obtained from one another by a finite sequence of isotopies in the class of transverse-Legendrian links,
and exchange moves.
\end{prop}

\begin{proof}
Due to Proposition~\ref{i+ii-prop}, proving the `if' part amounts to checking that exchange moves
of transverse-Legendrian links can be realized my means of elementary moves of respective rectangular diagrams.
We leave this to the reader, and don't use in the sequel.

To prove the `only if' part, first, note that every elementary move of rectangular diagrams can be
decomposed into a sequence of `even more elementary' ones, namely, such that each of the annuli~$(\theta_1;\theta_2)\times\mathbb S^1$
and~$\mathbb S^1\times(\varphi_1;\varphi_2)$ (we use the notation from Definition~\ref{moves-def})
contains at most one edge or the diagram being transformed. This follows from the fact that a single elementary move
associated with the rectangle~$[\theta_1;\theta_2]\times[\varphi_1;\varphi_2]$ can be decomposed
into two moves associated with rectangles~$[\theta_1;\theta_3]\times[\varphi_1;\varphi_2]$
and~$[\theta_3;\theta_2]\times[\varphi_1;\varphi_2]$ (respectively, $[\theta_1;\theta_2]\times[\varphi_1;\varphi_3]$ and
$[\theta_1;\theta_2]\times[\varphi_3;\varphi_2]$) for any~$\theta_3\in(\theta_1;\theta_2)$
(respectively, $\varphi_3\in(\varphi_1;\varphi_2)$)
such that the meridian~$\{\theta_2\}\times\mathbb S^1$ (respectively, the longitude~$\mathbb S^1\times\{\varphi_3\}$)
contains no vertices of the diagram.

In each case of an `even more elementary' move (there are now only finitely many to consider),
it is a direct check that the corresponding transverse-Legendrian
link undergoes an isotopy in the class of transverse-Legendrian links, possibly composed with an exchange move.
\end{proof}

\section{Applications}
Corollary~\ref{main-coro} gives a theoretical solution of the equivalence problem for
transverse links having trivial orientation-preserving symmetry group, but
an implementation of the algorithm takes a lot of time in general.
However, the results of~\cite{distinguishing,dyn-shast} supplemented by Propositions~\ref{R3-prop} and~\ref{class-prop} above
allow, in some cases, to distinguish transverse knots having trivial orientation-preserving symmetry group
with very little effort. To illustrate this, we consider the knots~$10_{128}$ and~$10_{160}$.

It is conjectured in~\cite{chong2013} chat the $\xi_+$-Legendrian knots
associated with the rectangular diagrams~$10_{128}^{1\mathrm R}$ and~$-\mu(10_{128}^{1\mathrm R})$
shown in Figure~\ref{128-r-fig} (we use the notation of~\cite{dyn-shast} for these diagrams)
\begin{figure}[ht]
\begin{tabular}{ccc}
\includegraphics[scale=.25]{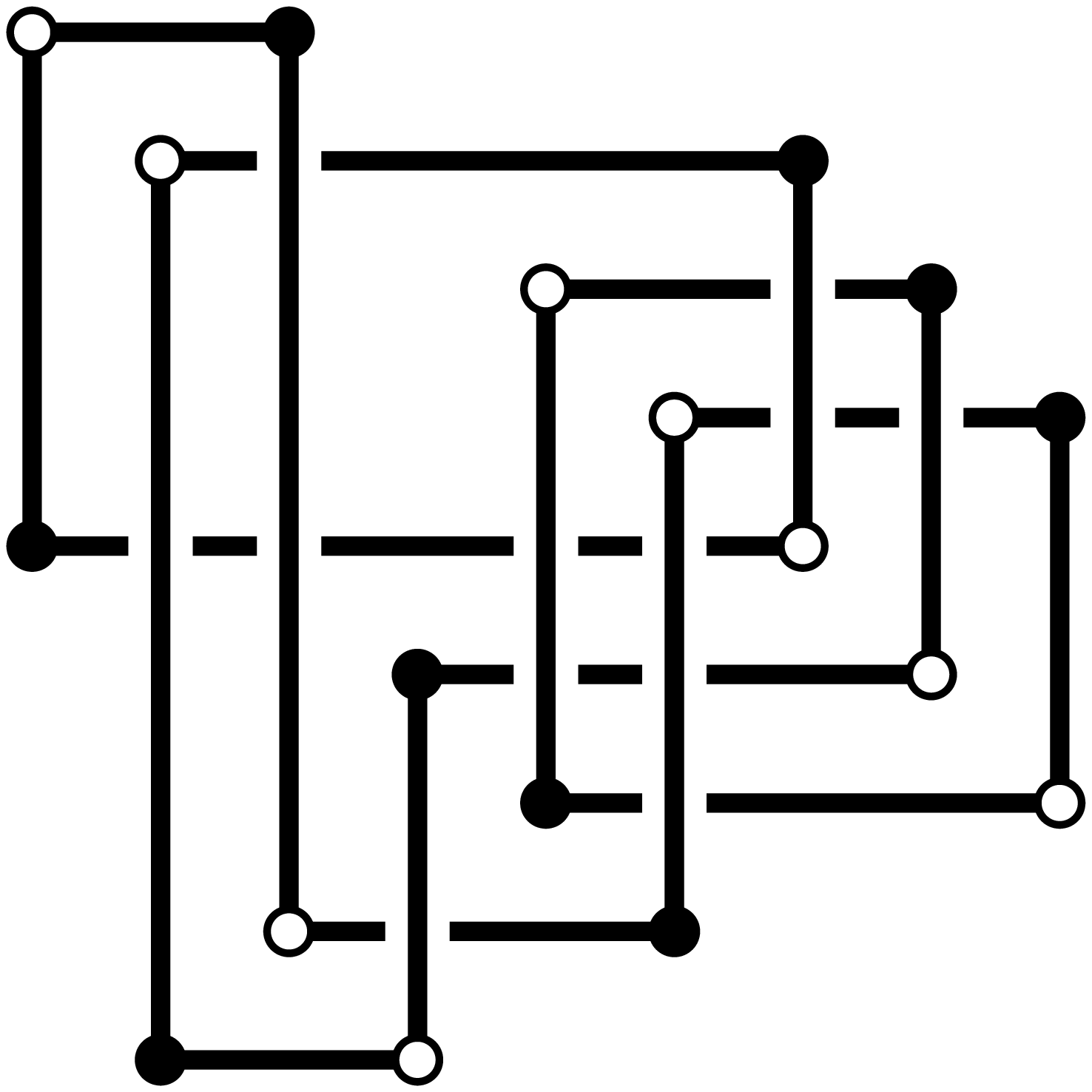}
&\hbox to 2cm{\hss}&
\includegraphics[scale=.25]{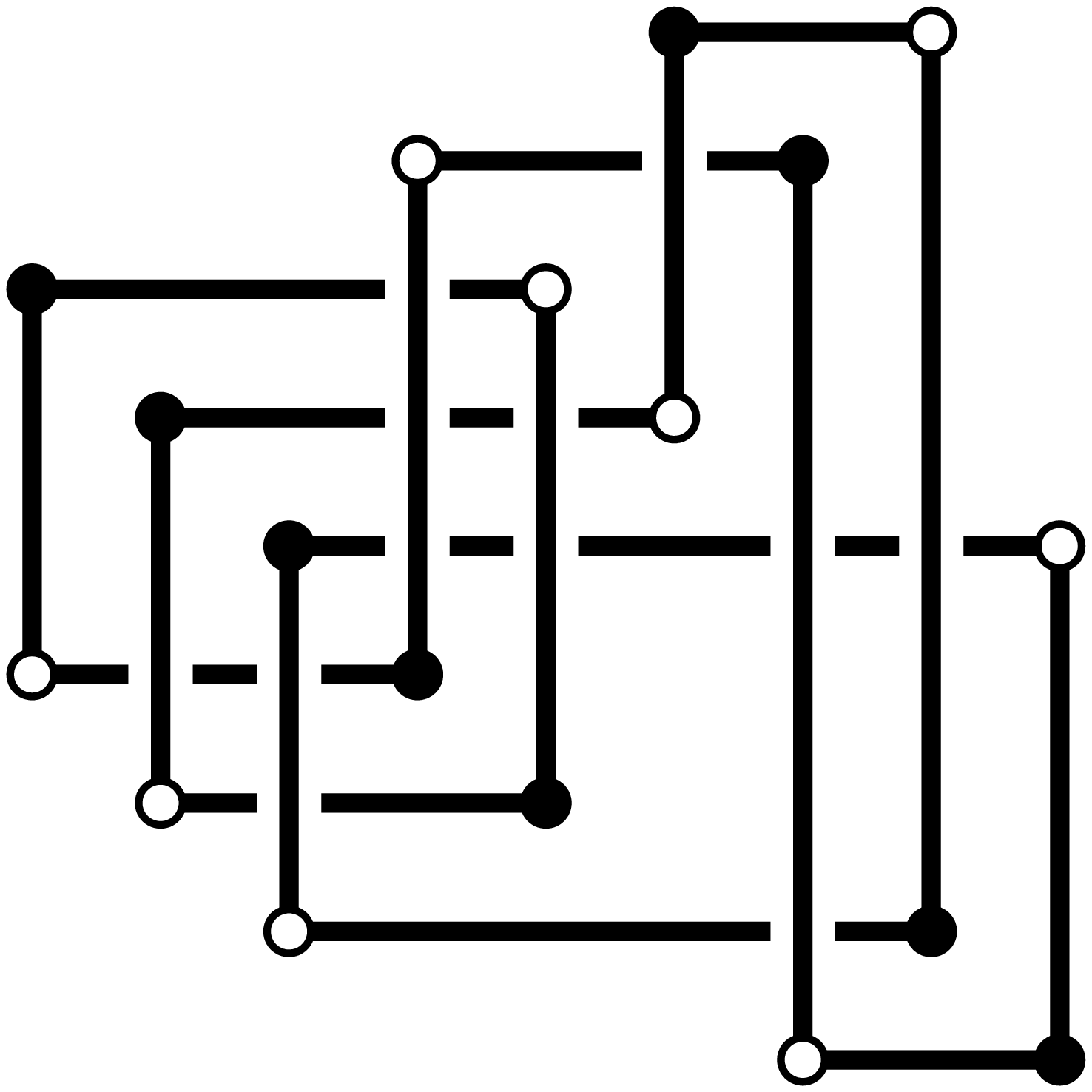}\\
$10_{128}^{1\mathrm R}$&&$-\mu(10_{128}^{1\mathrm R})$
\end{tabular}
\caption{The diagrams~$10_{128}^{1\mathrm R}$ and~$-\mu(10_{128}^{1\mathrm R})$}\label{128-r-fig}
\end{figure}
are not Legendrian isotopic, and, moreover, remain such after
any number of negative stabilizations. This is equivalent to saying that the positively $\xi_+$-transverse
knots associated with these diagrams are not transversely isotopic.
In the notation introduced in the beginning of this paper,
this inequality can also be written as
\begin{equation}\label{128-transverse-neq}
[10_{128}^{1\mathrm R}]_{\overrightarrow{\mathrm I},\overleftarrow{\mathrm I},\overrightarrow{\mathrm{II}}}
\ne[-\mu(10_{128}^{1\mathrm R})]_{\overrightarrow{\mathrm I},\overleftarrow{\mathrm I},\overrightarrow{\mathrm{II}}}.
\end{equation}

This conjecture was partially confirmed in~\cite[Proposition~7.5]{dyn-shast},
namely, it was shown that the Legendrian knots in questions are, indeed,
not equivalent, and remain such after up to four negative stabilizations. Extending
this to any number of negative stabilizations now amounts to showing that
\begin{equation}\label{128-ne-eq}
[10_{128}^{1\mathrm R}]_{\overrightarrow{\mathrm{II}}}\ne
[-\mu(10_{128}^{1\mathrm R})]_{\overrightarrow{\mathrm{II}}}.
\end{equation}

The type~$\overrightarrow{\mathrm{II}}$ transverse-Legendrian knots associated with
the diagrams~$10_{128}^{1\mathrm R}$ and~$-\mu(10_{128}^{1\mathrm R})$ are shown in Figure~\ref{torus-128-fig}.
\begin{figure}[ht]
\begin{tabular}{ccc}
\includegraphics[scale=.25]{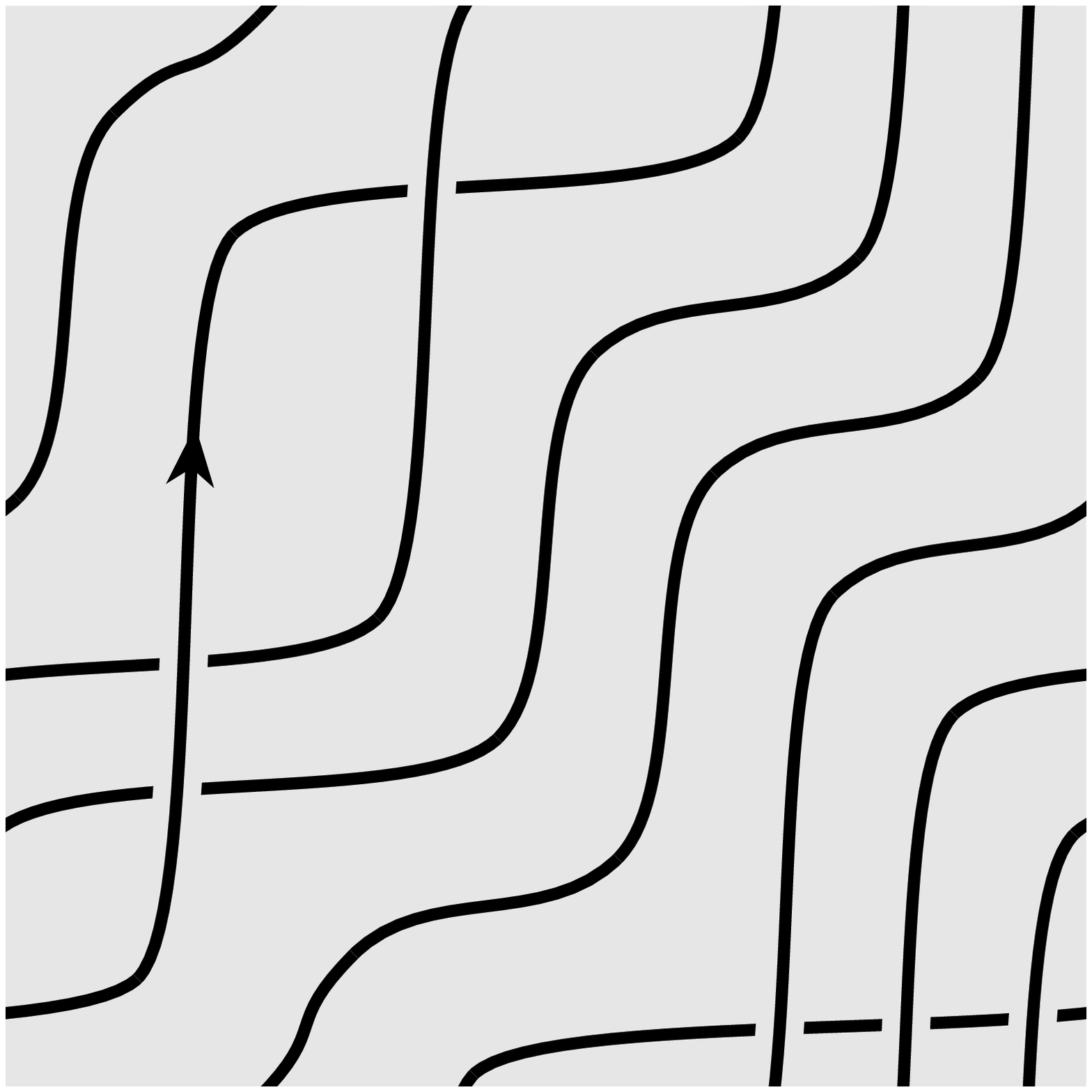}
&\hbox to 2cm{\hss}&
\includegraphics[scale=.25]{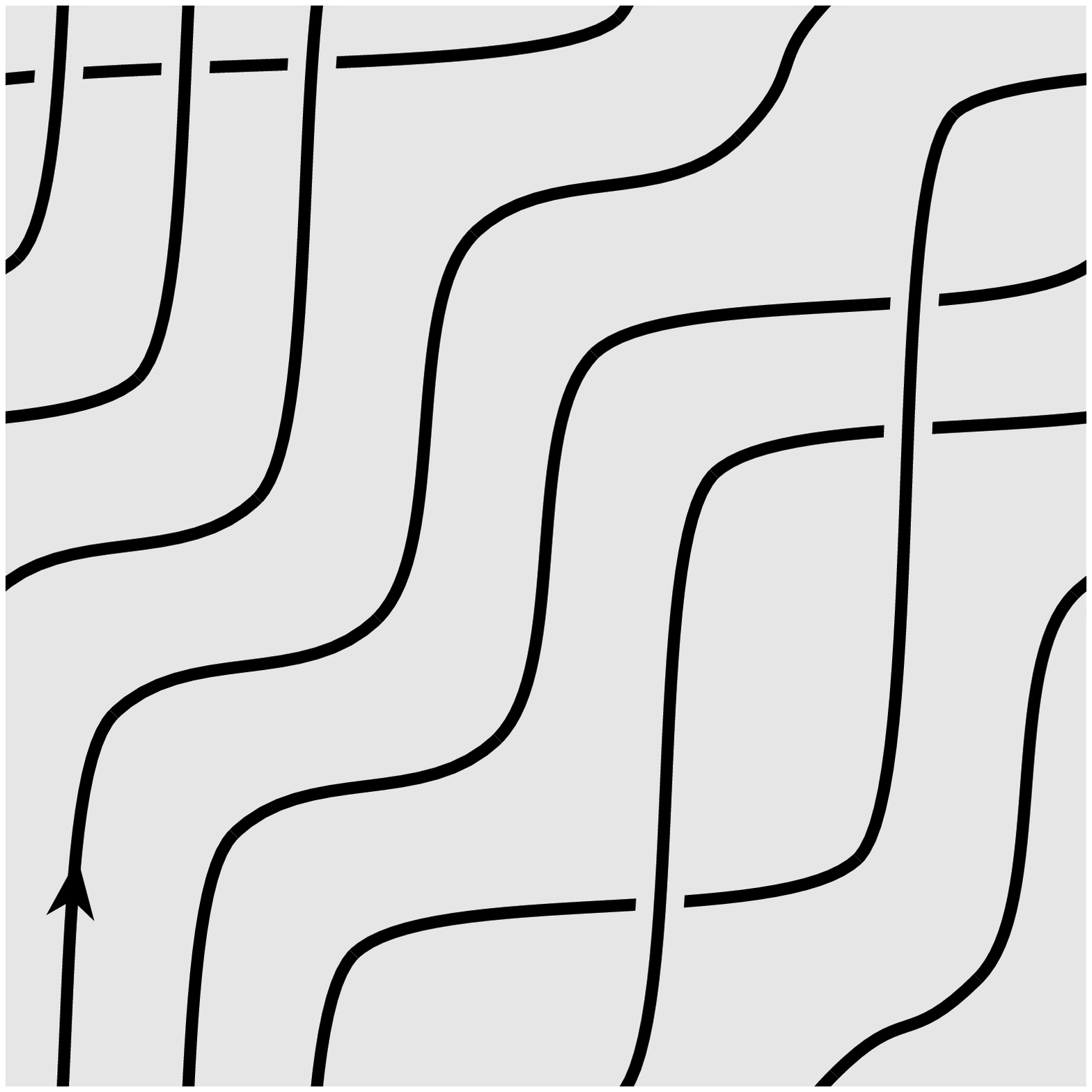}\\
$\mathrm{TL}_{\overrightarrow{\mathrm{II}}}(10_{128}^{1\mathrm R})$&&$\mathrm{TL}_{\overrightarrow{\mathrm{II}}}(-\mu(10_{128}^{1\mathrm R}))$
\end{tabular}
\caption{Torus projections of knots from $\mathrm{TL}_{\protect\overrightarrow{\mathrm{II}}}(10_{128}^{1\mathrm R})$
and~$\mathrm{TL}_{\protect\overrightarrow{\mathrm{II}}}(-\mu(10_{128}^{1\mathrm R}))$}\label{torus-128-fig}
\end{figure}
There are no `triangles' in the complement of any of these torus fronts, hence no Reidemeister-III move
can be applied to them. It is also not hard to see that these torus fronts admit no exchange moves,
even after any isotopy in the class of positive torus fronts, and that they are not isotopic.
By Propositions~\ref{R3-prop} and~\ref{class-prop}, this implies~\eqref{128-ne-eq}, and then~\eqref{128-transverse-neq} by~\cite[Theorem~4.2 and Figure~20]{dyn-shast}.

Thus, we have the following.

\begin{prop}
The positively $\xi_+$-transverse
knots associated with the diagrams~$10_{128}^{1\mathrm R}$ and~$-\mu(10_{128}^{1\mathrm R})$
are not transversely isotopic.
\end{prop}

In a completely similar fashion the following statement, which also confirms a conjecture of~\cite{chong2013}, is established.

\begin{prop}
The positively $\xi_+$-transverse
knots associated with the diagrams~$-10_{160}^{2\mathrm R}$ and~$10_{160}^{3\mathrm R}$
shown in Figure~\ref{160-r-fig}
are not transversely isotopic.
\end{prop}
\begin{figure}[ht]
\begin{tabular}{ccc}
\includegraphics[scale=.25]{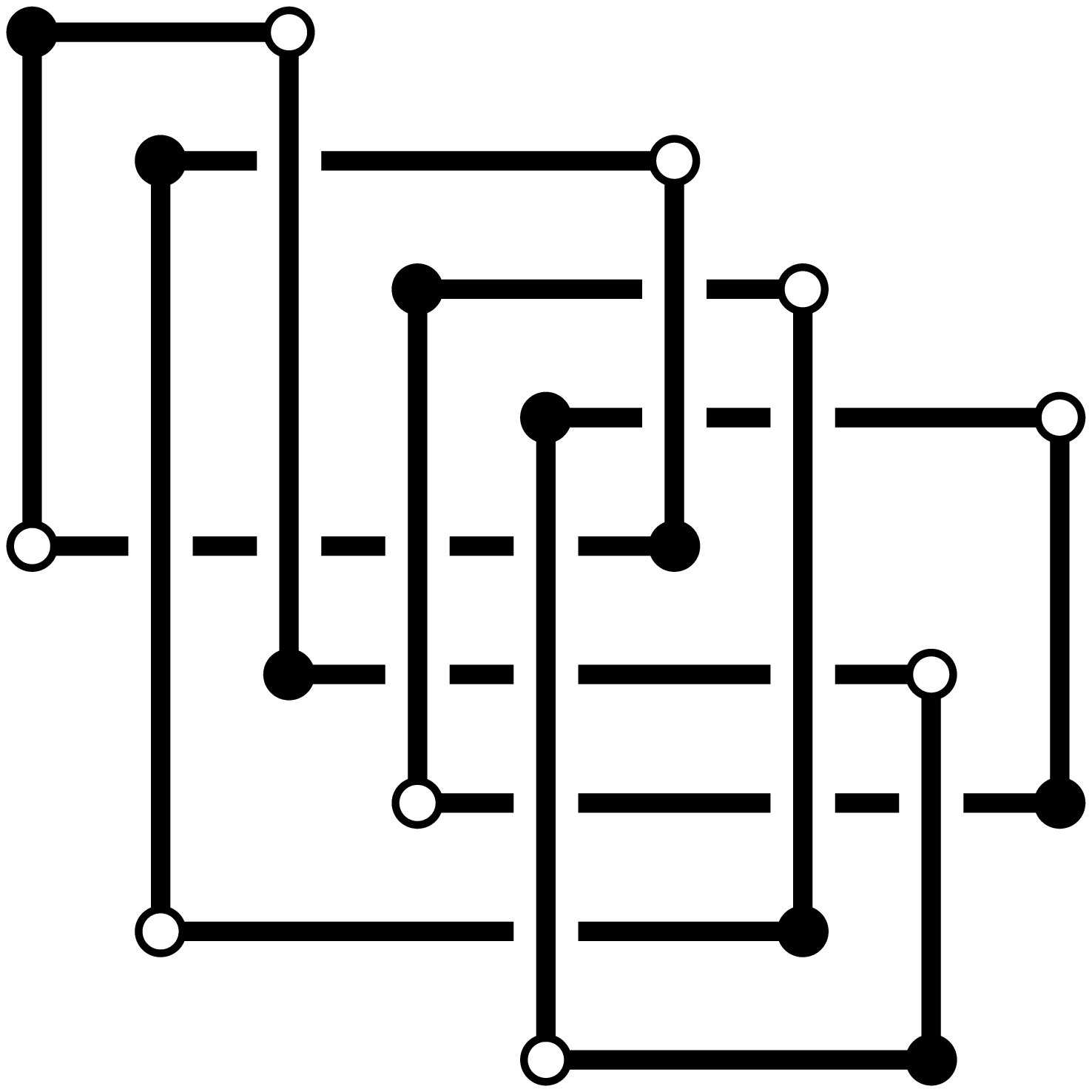}
&\hbox to 2cm{\hss}&
\includegraphics[scale=.25]{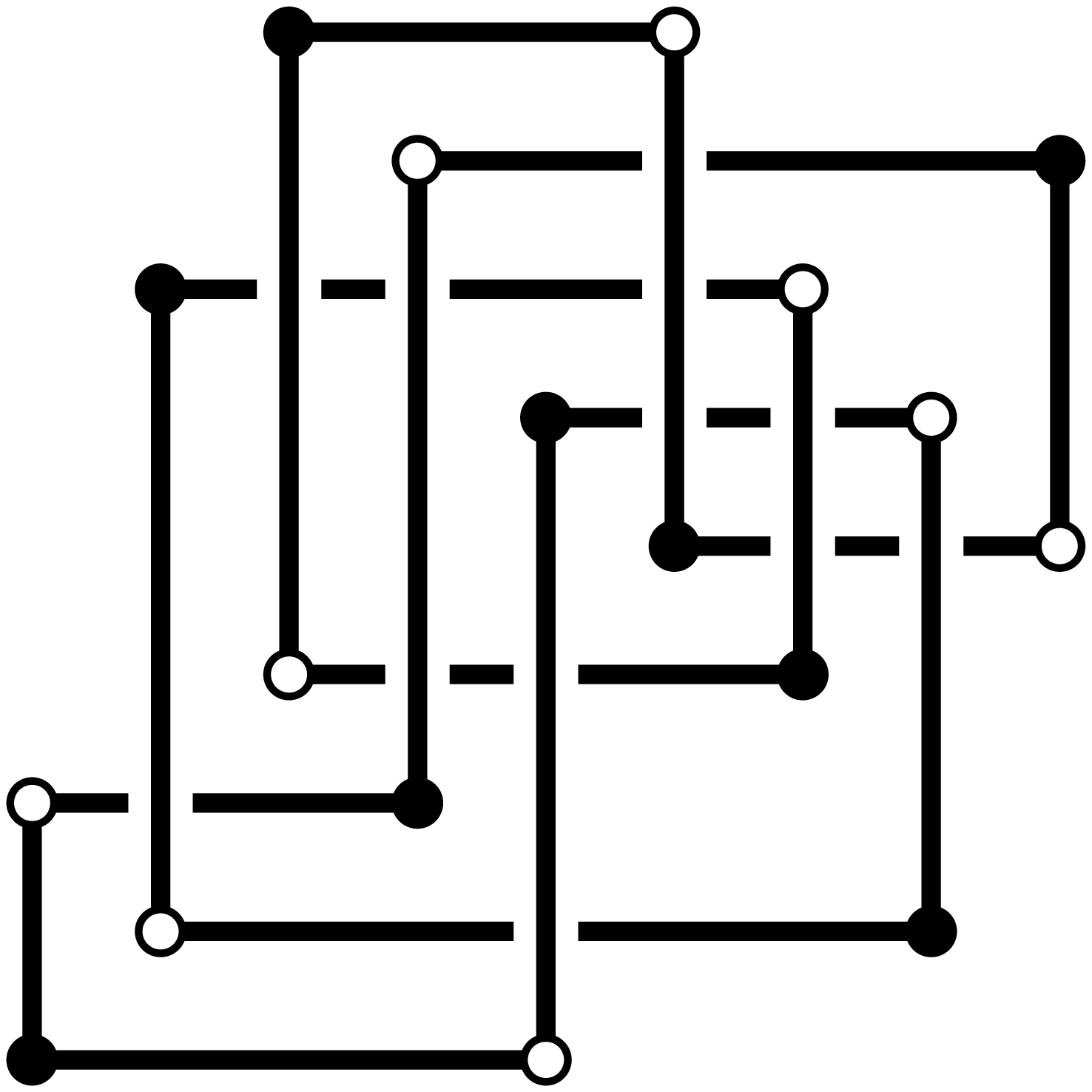}\\
$-10_{160}^{2\mathrm R}$&&$10_{160}^{3\mathrm R}$
\end{tabular}
\caption{The diagrams~$-10_{160}^{2\mathrm R}$ and~$10_{160}^{3\mathrm R}$}\label{160-r-fig}
\end{figure}

The proof is obtained by analyzing the torus projections in Figure~\ref{torus-160-fig} (see \cite[Figure~22]{dyn-shast}
for the notation and a description of the relation between these diagrams and those in~\cite{chong2013}).

\begin{figure}
\begin{tabular}{ccc}
\includegraphics[scale=.25]{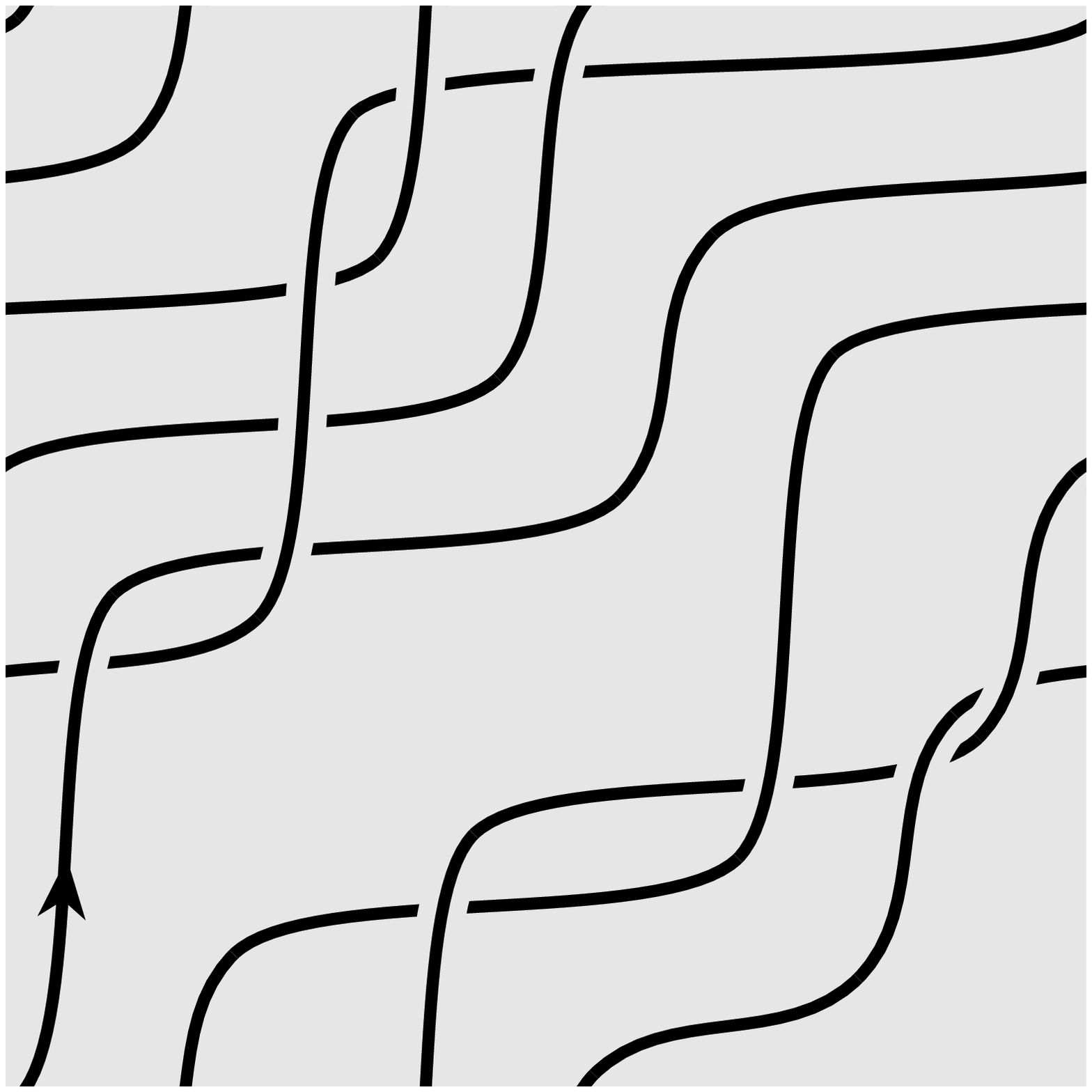}
&\hbox to 2cm{\hss}&
\includegraphics[scale=.25]{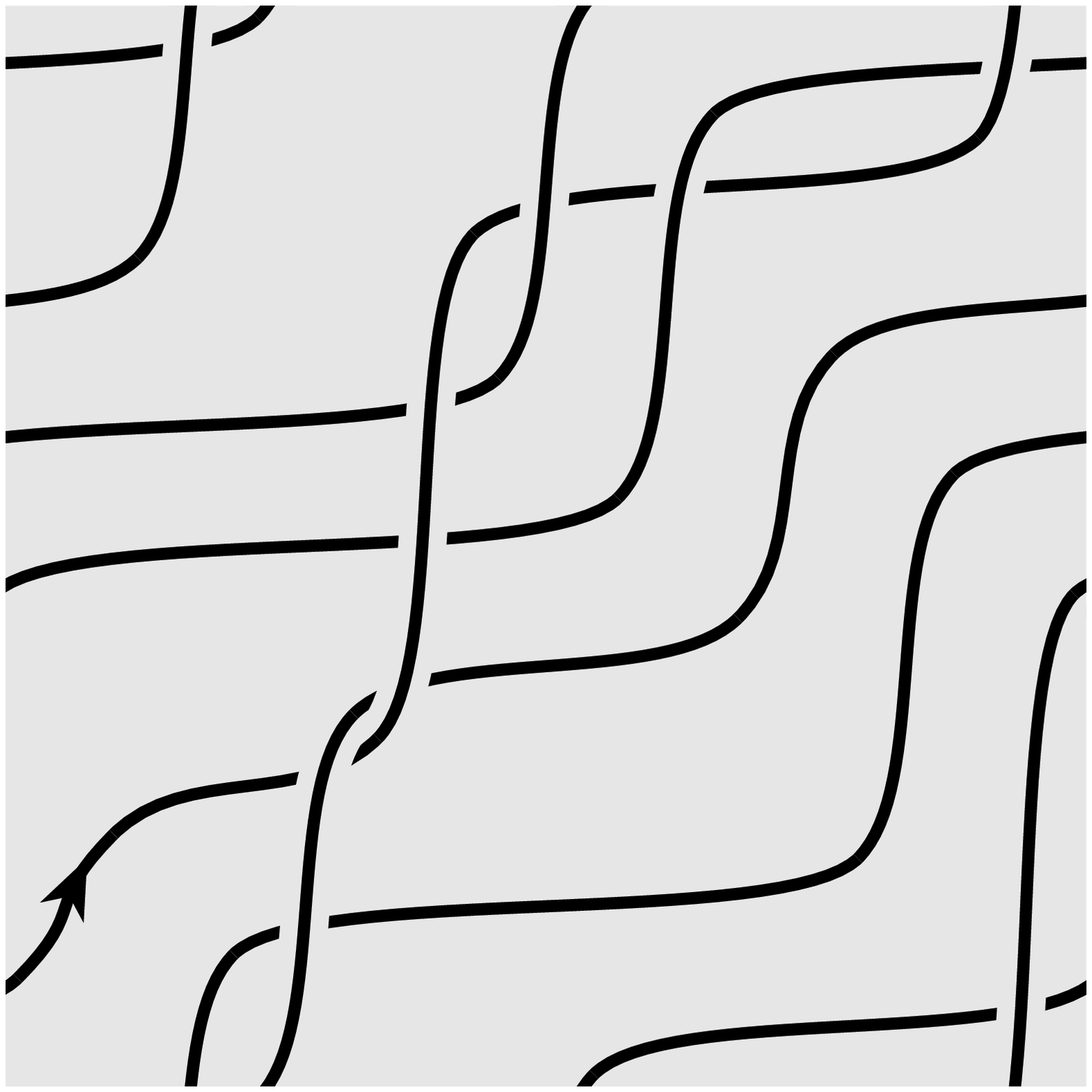}\\
$\mathrm{TL}_{\protect\overrightarrow{\mathrm{II}}}(-10_{160}^{2\mathrm R})$&&$\mathrm{TL}_{\protect\overrightarrow{\mathrm{II}}}(10_{160}^{3\mathrm R})$
\end{tabular}
\caption{Torus projections of $\mathrm{TL}_{\protect\overrightarrow{\mathrm{II}}}(-10_{160}^{2\mathrm R})$
and~$\mathrm{TL}_{\protect\overrightarrow{\mathrm{II}}}(10_{160}^{3\mathrm R})$}\label{torus-160-fig}
\end{figure}

\section{Concluding remarks}
Four oriented types of stabilizations and destabilizations of rectangular diagrams of links
are symmetric to each other and play equal roles in knot theory. This means that with every rectangular diagram~$R$ of a link
one can associate four different objects having the nature of a transverse-Legendrian link type:
\begin{itemize}
\item
a positively $\xi_+$-transverse and $\xi_-$-Legendrian link type, which is identified with~$[R]_{\overrightarrow{\mathrm{II}}}$,
\item
a negatively $\xi_+$-transverse and $\xi_-$-Legendrian link type, which is identified with~$[R]_{\overleftarrow{\mathrm{II}}}$,
\item
a positively $\xi_-$-transverse and $\xi_+$-Legendrian link type, which is identified with~$[R]_{\overrightarrow{\mathrm I}}$, and
\item
a negatively $\xi_-$-transverse and $\xi_+$-Legendrian link type, which is identified with~$[R]_{\overleftarrow{\mathrm I}}$.
\end{itemize}
This is illustrated in Figure~\ref{four-tl-fig}, where torus projections of all four transverse-Legendrian links are shown.

\begin{figure}[ht]
\centerline{\includegraphics[scale=.2]{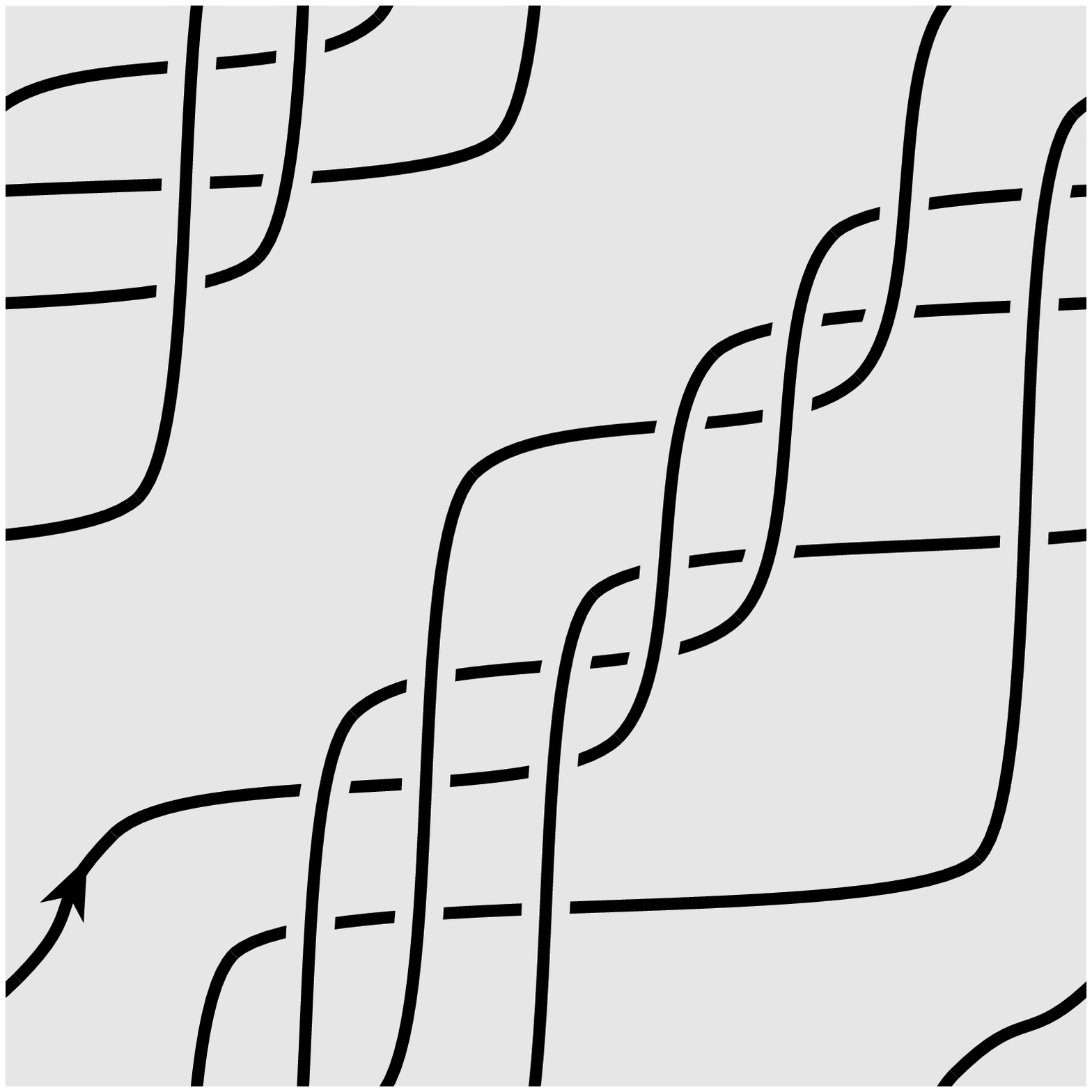}\smash{\put(-63,-15){$\mathrm{TL}_{\overrightarrow{\mathrm{II}}}(R)$}}\hskip6cm
\includegraphics[scale=.2]{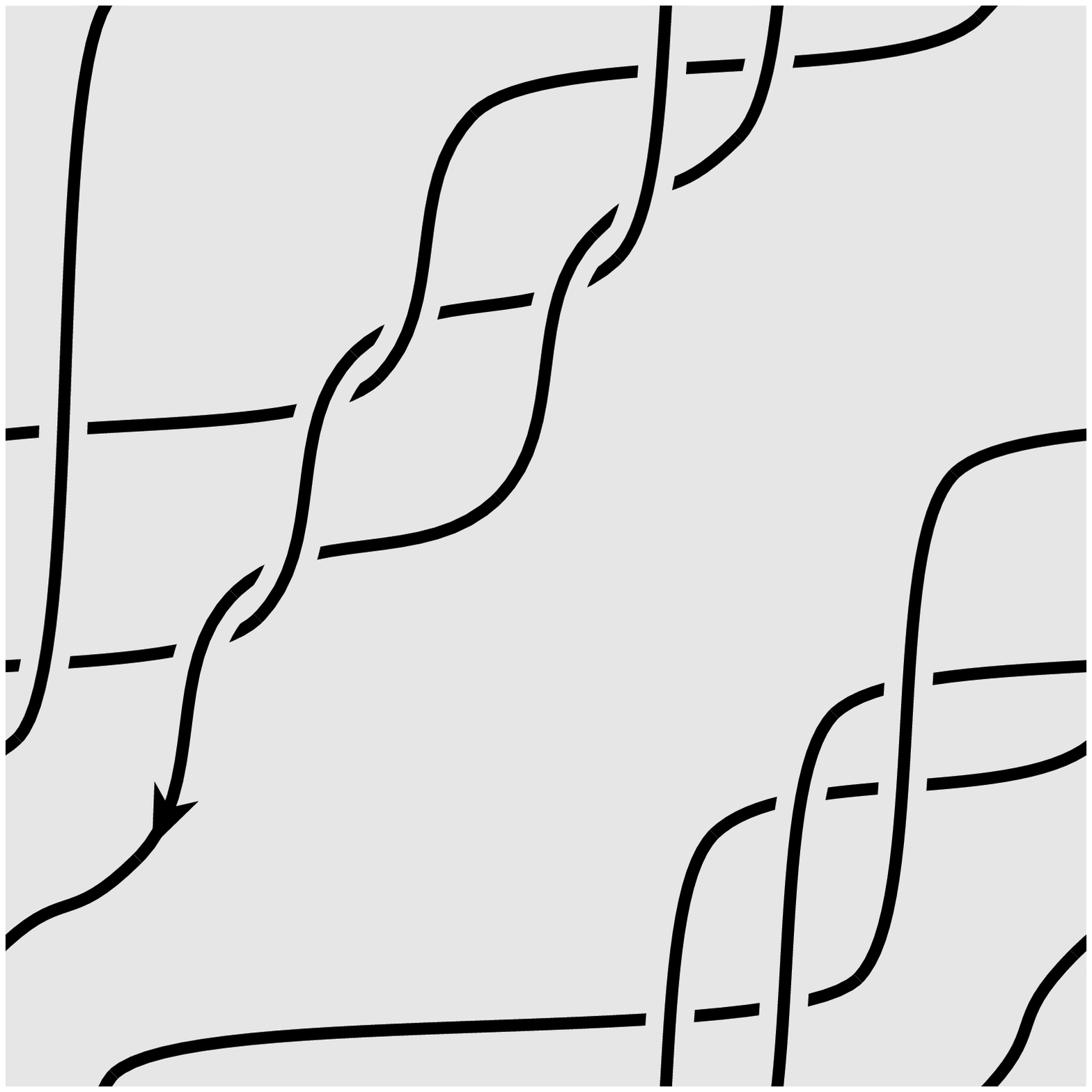}\smash{\put(-63,-15){$\mathrm{TL}_{\overleftarrow{\mathrm{II}}}(R)$}}}

\vskip-.5cm 

\centerline{\includegraphics[scale=.2]{rd1.eps}\smash{\put(-50,-15){$R$}}}

\vskip-.5cm 

\centerline{\includegraphics[scale=.2]{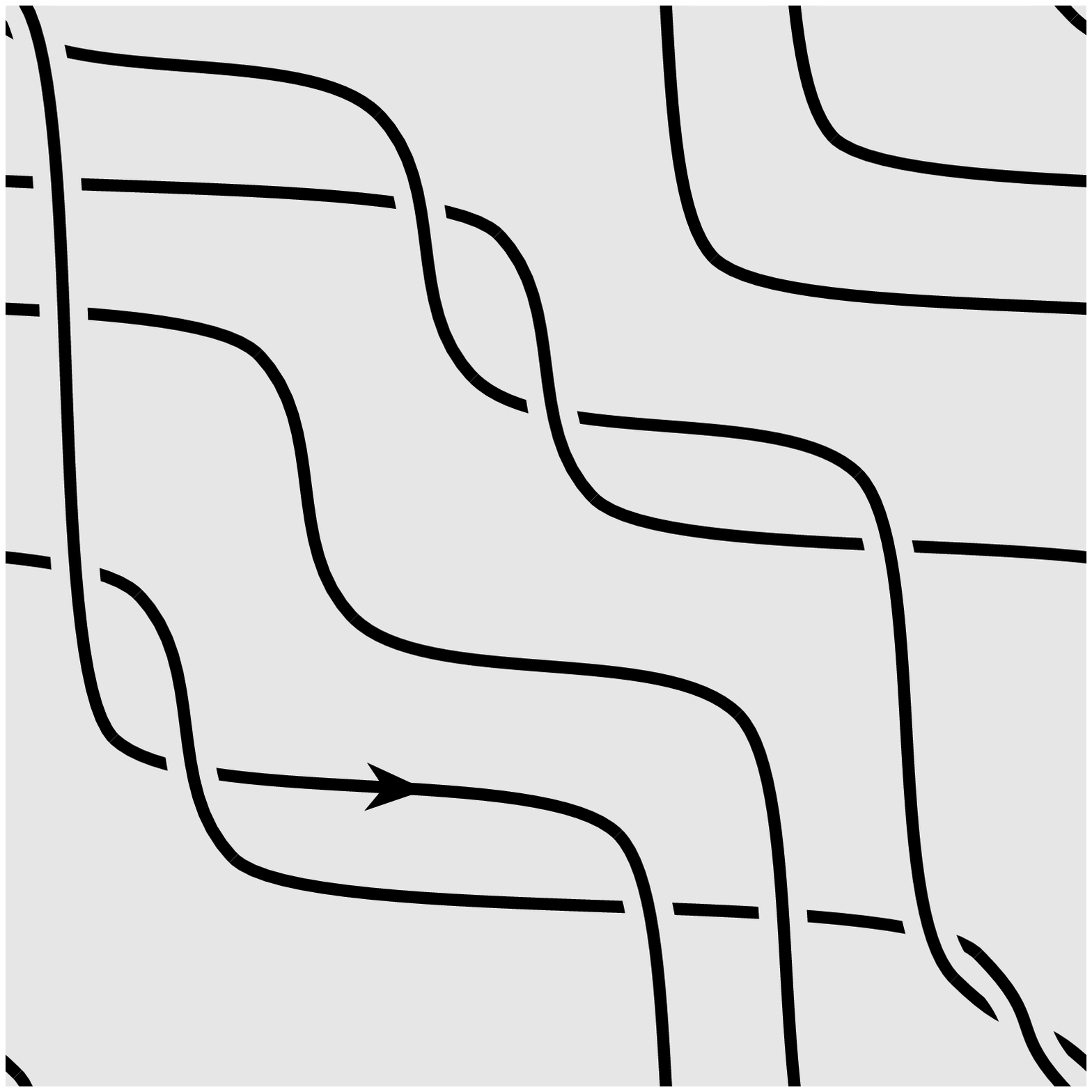}\smash{\put(-63,-15){$\mathrm{TL}_{\overrightarrow{\mathrm I}}(R)$}}\hskip6cm
\includegraphics[scale=.2]{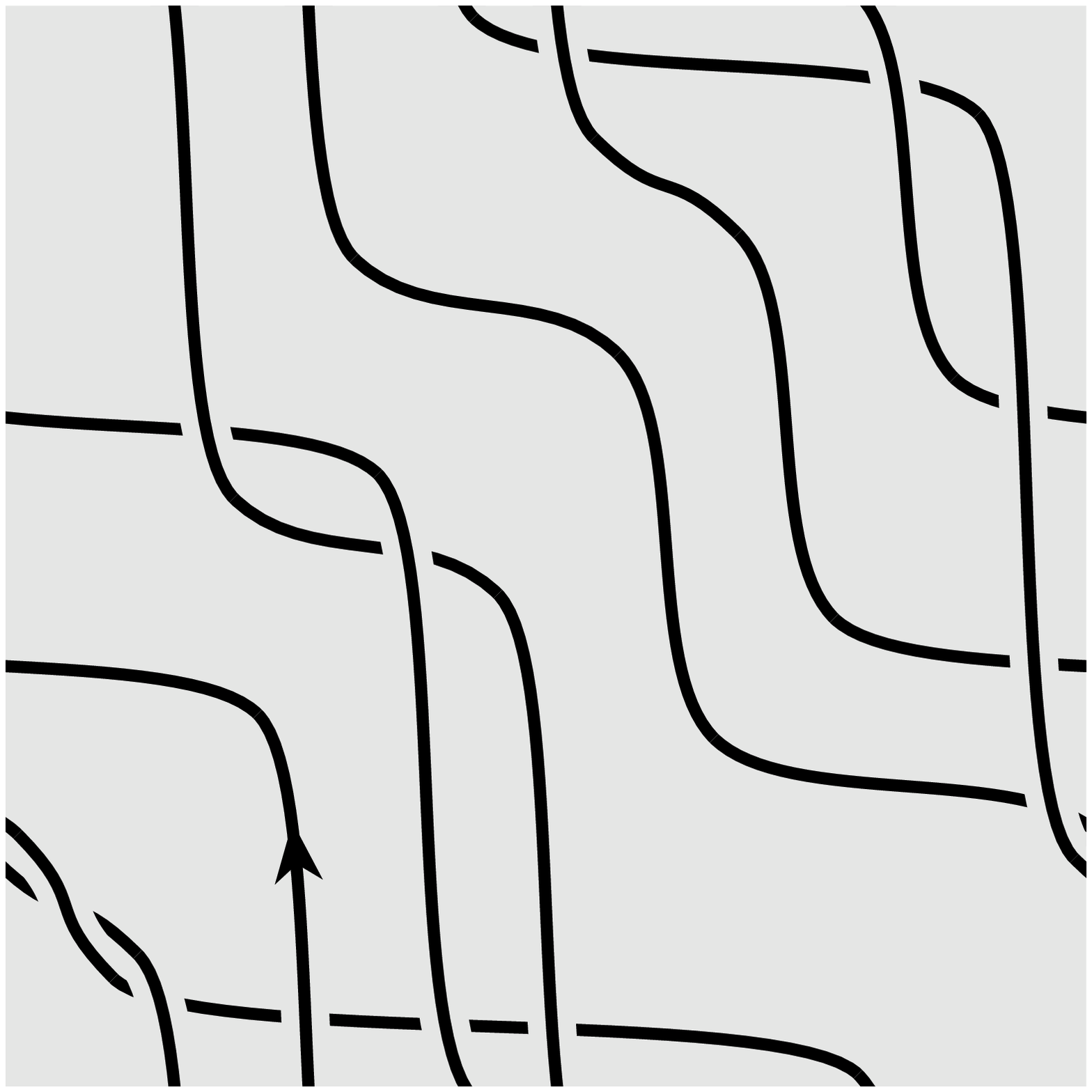}\smash{\put(-63,-15){$\mathrm{TL}_{\overleftarrow{\mathrm I}}(R)$}}}

\vskip5mm
\caption{Four transverse-Legendrian links associated with a single rectangular diagram}\label{four-tl-fig}
\end{figure}

One may naturally ask if there is a relation between rectangular diagrams and links which are Legendrian
with respect to both contact structures~$\xi_+$ and~$\xi_-$, or transverse with respect to both of them.
The answer in both cases is pretty simple. Links that are $\xi_+$-Legendrian and~$\xi_-$-Legendrian simultaneously
are exactly the links of the form~$\widehat R$, where~$R$ is a rectangular diagram of a link (one should
extend the definition of a Legendrian link to piecewise smooth curves, since the links of the form~$\widehat R$ are
typically non-smooth). So, equivalence classes of such links are in one-to-one correspondence with
combinatorial types of rectangular diagrams.

Links which are positively $\xi_+$-transverse and positively $\xi_-$-transverse are nothing else but closed
braids with~$\mathbb S^1_{\tau=0}$ as the axis. The isotopy classes of such links are the same
thing as conjugacy classes of braids. As noted in the beginning of the paper,
rectangular diagrams allow to classify braids modulo conjugacy and
Birman--Menasco exchange moves (as the elements of~$\mathscr R/\langle\overrightarrow{\mathrm I},\overrightarrow{\mathrm{II}}\rangle$). 
The situation with transverse-Legendrian links reflected in Proposition~\ref{class-prop}
is completely analogues.

\end{document}